\documentclass{amsart}
\usepackage{amssymb}
\usepackage{graphicx}
\usepackage[shortlabels]{enumitem}
\usepackage{multirow}
\usepackage{amsmath,color}
\usepackage{hyperref}
\usepackage{url}
\usepackage{longtable}

\theoremstyle{plain}
\newtheorem{theorem}{Theorem}[section]
\newtheorem{prop}[theorem]{Proposition}
\newtheorem{proposition}[theorem]{Proposition}
\newtheorem{lemma}[theorem]{Lemma}
\newtheorem{corollary}[theorem]{Corollary}

\theoremstyle{definition}

\theoremstyle{remark}
\newtheorem{remark}[theorem]{Remark}

{

\newcommand{\inprod}[1]{\langle#1\rangle}

\newcommand{\Z}{\mathbb{Z}}
\newcommand{\regSub}{\Lambda}
\newcommand{\SRG}{\operatorname{srg}}
\newcommand{\GDD}{\operatorname{gdd}}
\newcommand{\AR}{\operatorname{ar}}

\begin{document}
\title{Biregular graphs with three eigenvalues}

\author[X.-M. Cheng]{ Xi-Ming Cheng}
\author[A. L. Gavrilyuk]{ Alexander L. Gavrilyuk$^\diamondsuit$ }
\thanks{$^\diamondsuit$ Part of the work was done while A.L.G. was visiting Tohoku University as a JSPS
Postdoctoral Fellow. His work (e.g., Theorem 7.3) was also partially supported by the Russian Science Foundation (grant 14-11-00061)}
\author[G. R. W. Greaves]{ Gary R. W. Greaves$^\clubsuit$ }
\thanks{$^\clubsuit$G.R.W.G. was supported by JSPS KAKENHI; 
grant number: 26$\cdot$03903}
\author[J. H. Koolen]{ Jack H. Koolen$^\spadesuit$ }
\thanks{$^\spadesuit$J.H.K. is partially supported by the `100 talents' program of the Chinese Academy of Sciences,
and by the National Natural Science Foundation of China (No. 11471009).}
\thanks{\copyright 2016. This manuscript version is made available under the CC-BY-NC-ND 4.0 license http://creativecommons.org/licenses/by-nc-nd/4.0/ }


\address{School of Mathematical Sciences, 
University of Science and Technology of China,
Hefei, Anhui, 230026, P.R. China}
\email{xmcheng@mail.ustc.edu.cn}

\address{Department of Algebra and Topology,
N.N. Krasovskii Institute of Mathematics and Mechanics UB RAS,
 Yekaterinburg, Russia}
\email{alexander.gavriliouk@gmail.com}

\address{Research Center for Pure and Applied Mathematics,
Graduate School of Information Sciences, 
Tohoku University, Sendai 980-8579, Japan}
\email{grwgrvs@gmail.com}

\address{Wen-Tsun Wu Key Laboratory of CAS, School of Mathematical Sciences, 
University of Science and Technology of China,
Hefei, Anhui, 230026, P.R. China}
\email{koolen@ustc.edu.cn}

\dedicatory{Dedicated to Misha Klin on the occasion of his retirement.}



\begin{abstract}
	We consider nonregular graphs having precisely three distinct eigenvalues.
	The focus is mainly on the case of graphs having two distinct valencies and our results include constructions of new examples, structure theorems, valency constraints, and a classification of certain special families of such graphs.
	We also present a new example of a graph with three valencies and three eigenvalues of which there are currently only finitely many known examples.
\end{abstract}

\maketitle

\section{Introduction}

In the late 1990s, as a generalisation of strongly regular graphs, attention was brought to the study of nonregular graphs whose adjacency matrices have precisely three distinct eigenvalues.
We continue this investigation focussing mainly on graphs having precisely two distinct valencies, so-called \textbf{biregular} graphs.
Muzychuk and Klin~\cite{MK} called such graphs `strongly biregular graphs'. 

An $n$-vertex graph with a vertex of valency $n-1$ is called a \textbf{cone}.
Given a graph $\Gamma$, the \textbf{cone over} $\Gamma$ is the graph formed by adjoining a vertex adjacent to every vertex of $\Gamma$.
Examples of families of strongly biregular graphs are complete bipartite graphs and cones over strongly regular graphs.
Indeed, a complete bipartite graph $K_{n,m}$ (for $n > m \geqslant 1$) has spectrum $\{ [\sqrt{nm}]^1, [0]^{n+m-2}, [-\sqrt{nm}]^1 \}$.
The following result due to Muzychuk and Klin offers a method for finding strongly biregular cones.

\begin{proposition}[See \cite{MK}]\label{propcone}
Let $\Gamma$ be a (non-complete) strongly regular graph with $n$ vertices, valency $k$, and smallest eigenvalue $\theta_2$. 
Then the cone over $\Gamma$ has three distinct eigenvalues if and only if $\theta_2(k- \theta_2) = -n$.
\end{proposition}

There are infinitely many strongly regular graphs satisfying the assumption of the proposition and so there are infinitely many cones over strongly regular graphs having three distinct eigenvalues \cite{brme,MK}.
As well as giving some sporadic examples, using symmetric and affine designs, Van Dam~\cite{Dam3ev} exhibited a couple of infinite families of strongly biregular graphs that are neither cones nor complete bipartite graphs.

So far we only have a finite list of graphs with three valencies and three distinct eigenvalues and no examples of graphs with precisely three distinct eigenvalues with more than three valencies.
Below we contribute a new graph to the list of graphs with three eigenvalues and three valencies whilst also showing the nonexistence of an, a priori, putative graph with three eigenvalues and four valencies (see Section~\ref{sec:computational_constructions}).

There exist some partial classifications of graphs having three distinct eigenvalues in the following senses.
Van Dam~\cite{Dam3ev} classified all such graphs having smallest eigenvalues at least $-2$ and also classified all such graphs on at most $29$ vertices.
We show the existence of some graphs on $30$ vertices (see Theorem~\ref{thm:30}) whose existence was an open question in \cite{Dam3ev}.
Chuang and Omidi~\cite{Chuang09} classified those graphs whose spectral radius is less than $8$.
We also contribute a classification of strongly biregular graphs whose second largest eigenvalue is at most $1$ (see Section~\ref{sec:sec1}).

In this paper we further develop the theory of graphs with precisely three distinct eigenvalues.
We begin with Section~\ref{sec:prelim} where we present our preliminaries, classify connected graphs that have three distinct eigenvalues and a disconnected complement.
In Section~\ref{sec:bounds_for_graphs_with_three_eigenvalues} we give some bounds for various quantities related to graphs with three distinct eigenvalues.
In Section~\ref{sec:bireg} we focus on non-bipartite strongly biregular graphs and, in particular, we show that there are only finitely many such graphs when the second largest eigenvalue is bounded.
Here we also define what we mean by \emph{feasible valency-arrays and spectra}.
We show a relationship with strongly biregular graphs and certain designs in Section~\ref{sec:graphs_from_designs}.
Section~\ref{sec:comp} is concerned with graphs with three distinct eigenvalues whose complements also have precisely three distinct eigenvalues and biregular graphs with three eigenvalues whose switchings (with respect to the valency partition) also have three distinct eigenvalues.
In Section~\ref{sec:computational_constructions}, using the star complement method, we show the existence of new graphs having precisely three distinct eigenvalues.
We also study the structure of strongly biregular graphs whose two smaller eigenvalues sum to $-1$.
We also establish some nonexistence results in Section~\ref{sub:some_nonexistence_results} and state some open problems in Section~\ref{sec:open_problems}.
Furthermore, as an appendix, we provide a table of feasible valency-arrays and spectra for biregular graphs having precisely three distinct eigenvalues.

\section{Graphs with three distinct eigenvalues}
\label{sec:prelim}

In this section we develop some basic theory for graphs with three distinct eigenvalues.
We assume the reader is familiar with the basic definitions and techniques of algebraic graph theory; one can use Godsil and Royle's book~\cite{God01} as a reference.

Let $\Gamma = (V,E)$ be a connected graph on $n$ vertices.
Recall that the adjacency matrix $A$ of $\Gamma$ is an $n \times n$ matrix whose $(i,j)$th entry, $A_{i,j}$, is $1$ if the $i$th vertex of $\Gamma$ is adjacent to the $j$th vertex of $\Gamma$ and $0$ otherwise.
We write $x \sim y$ if the vertices $x$ and $y$ are adjacent and $x \not \sim y$ if they are not adjacent.
By the eigenvalues of $\Gamma$ we mean the eigenvalues of $A$.
Examining the trace of $A$ and $A^2$ gives two basic facts about the eigenvalues $\eta_i$ ($1 \leqslant i \leqslant n$):
\begin{equation}\label{sums} 
	\sum_{i=1}^{n} \eta_i = 0
	\quad \text{ and } \quad 
	\sum_{i=1}^{n} \eta_i^2 = 2e,
\end{equation}
where $e$ is the number of edges of $\Gamma$.

Assume that $\Gamma$ has precisely three distinct eigenvalues $\theta_0 > \theta_1 > \theta_2$. 
Then $\Gamma$ has diameter two and since such a graph cannot be complete, it follows by interlacing that $\theta_1 \geqslant 0$ and $\theta_2 \leqslant -\sqrt{2}$.

By the Perron-Frobenius Theorem (see, for example, \cite{God01}), $\theta_0$ has multiplicity 1 and the entries of any eigenvector for $\theta_0$ have constant sign.
This implies that there exists a positive eigenvector ${\mathbf \alpha}$ for the eigenvalue $\theta_0$ such that 
$$
(A-\theta_1 I)(A - \theta_2I) = {\mathbf \alpha} {\mathbf \alpha}^T.
$$
Throughout the paper we exclusively reserve the symbol $\alpha$ to correspond to this eigenvector.
For a vertex $x$, denote the entry of $\alpha$ corresponding to $x$ by $\alpha_x$. 
This implies that if a vertex $x$ has valency $d_x$, then $d_x = \alpha_x^2 - \theta_1 \theta_2$.
Let $x$ and $y$ be distinct vertices of $\Gamma$.
We write $\nu_{x,y}$ for the number of common neighbours of $x$ and $y$.
By the above formulae we have
\[
	\nu_{x,y} = (\theta_1 + \theta_2) A_{x,y} + \alpha_x \alpha_y.
\]
Assume that $\Gamma$ has $s$ distinct valencies $k_1, \dots, k_s$. 
We will often abuse our above notation by writing $\alpha_i$ to mean $\alpha_x$ for some vertex $x$ having valency $k_i$.
We may also write $\nu_{i,j}$ to mean $\nu_{x,y}$ where $d_x = k_i$ and $d_y = k_j$.
Throughout the paper, we will assume this notation to be standard.

We write $m_i$ for the multiplicity of eigenvalue $\theta_i$ of $\Gamma$. 
If $\Gamma$ has $n$ vertices then, since $1+m_1 + m_2 = n$ and $\theta_0 + m_1 \theta_1 + m_2 \theta_2=0$, we have
\begin{equation}\label{mult} 
	m_1 = -\frac{ (n-1)\theta_2 + \theta_0}{\theta_1 -\theta_2}
	\quad \text{ and } \quad 
	m_2 = \frac{(n-1) \theta_1 + \theta_0}{\theta_1 - \theta_2}.
\end{equation}

We call a graph \textbf{strongly regular} if it is a connected regular graph with constants $\lambda$ and $\mu$ such that every pair of vertices has $\lambda$ or $\mu$ common neighbours if they are adjacent or non-adjacent, respectively.
We use the notation $\operatorname{srg}(n,k,\lambda,\mu)$ to denote such graphs with valency $k$ and $n$ vertices.
If $\Gamma$ is regular then, since it has precisely three distinct eigenvalues, it is well-known that $\Gamma$ must be strongly regular \cite{Dam3ev}.
In this paper we focus on the less well studied case of when $\Gamma$ is nonregular.

Recall we assumed that $\Gamma$ has $s$ distinct valencies $k_1, \dots, k_s$.  
We write $V_i := \{ v \in V(\Gamma) \; | \; d_v = k_i \}$ and $n_i := |V_i|$ for $i \in \{1,\dots,s\}$.
Clearly the subsets $V_i$ partition the vertex set of $\Gamma$.
We call this partition the \textbf{valency partition} of $\Gamma$.
Let $\pi = \{\pi_1,\dots,\pi_s\}$ be a partition of the vertices of $\Gamma$.
For each vertex $x$ in $\pi_i$, write $d_{x}^{(j)}$ for the number of neighbours of $x$ in $\pi_j$.
Then we write $b_{ij} = 1/|\pi_i|\sum_{x \in \pi_i} d_{x}^{(j)}$ for the average number of neighbours in $\pi_j$ of vertices in $\pi_i$.
The matrix $B_\pi := (b_{ij})$ is called the \textbf{quotient matrix} of $\pi$ and $\pi$ is called \textbf{equitable} if for all $i$ and $j$, we have $d_{x}^{(j)} = b_{ij}$ for each $x \in \pi_i$.
We will use repeatedly properties of the quotient matrices of partitions of the vertex set of a graph and we refer the reader to Godsil and Royles' book~\cite[Chapter 9]{God01} for the necessary background on equitable partitions and interlacing.

For fixed $\theta_0 > \theta_1 > \theta_2$, define the set $\mathcal G(\theta_0,\theta_1,\theta_2)$ of connected nonregular graphs having precisely three distinct eigenvalues $\theta_0$, $\theta_1$, and $\theta_2$.

Among graphs with three eigenvalues, complete bipartite graphs are distinguished in the following way.

\begin{theorem}[Proposition 2 \cite{Dam3ev}]\label{thm:integralsr}
	Let $\Gamma$ be a graph in $\mathcal G(\theta_0,\theta_1,\theta_2)$ where $\theta_0$ is not an integer.
	Then $\Gamma$ is a complete bipartite graph.
\end{theorem}

It was shown by Smith~\cite{smith70} that if the second largest eigenvalue of a connected graph $\Gamma$ is at most $0$ then $\Gamma$ is a complete $r$-partite graph with parts of size $p_1,\dots,p_r$, denoted $K_{p_1,\dots,p_r}$.
We will see below that complete bipartite graphs are the only nonregular multipartite graphs with precisely three distinct eigenvalues.

\begin{theorem}\label{thm:disconn}
	Let $\Gamma$ be a graph in $\mathcal G(\theta_0,\theta_1,\theta_2)$.
	If the complement of $\Gamma$ is disconnected, then $\Gamma$ is a cone or $\Gamma$ is complete bipartite.
\end{theorem}

\begin{proof}
Let $V$ be the vertex set of $\Gamma$ and suppose that the complement $\overline{\Gamma}$ has at least $2$ connected components.
\paragraph{\textbf{Claim 1}} $\Gamma$ has at most three valencies. 
\label{par:claim1}

Suppose that $x$ and $y$ are vertices in different components of $\overline{\Gamma}$.
Then we must have $x \sim y$ (in $\Gamma$) and, since the other $n-2$ vertices must be adjacent to $x$ or $y$, we have $\nu_{x,y} = d_x+d_y-n$.
Hence we can write 
\begin{equation}
	\label{eqn:quad}
	\alpha_x^2-\alpha_x\alpha_y+\alpha_y^2 = n + \theta_1+\theta_2 + 2\theta_1\theta_2.
\end{equation}
If $d_x = d_y$ then, by Eq.~\eqref{eqn:quad}, $\alpha_x = \alpha_y = \sqrt{n + \theta_1+\theta_2 + 2\theta_1\theta_2}$.
Furthermore, since any other vertex $z$ of $\Gamma$ must be adjacent to $x$ or $y$, by Eq.~\eqref{eqn:quad}, we would have $\alpha_z = \alpha_y = \alpha_x$.
But this cannot happen since $\Gamma$ is not regular.
Hence, vertices in different components cannot have the same valency.

From Eq.~\eqref{eqn:quad} observe that if $\alpha_x$ is fixed then, there are only two possible values for $\alpha_y$, say $\alpha$ and $\alpha'$, satisfying $\alpha + \alpha' = \alpha_x$.
Thus, for any vertex $x$ in one connected component $C$ of $\overline \Gamma$, there can be at most two vertices outside of $C$ having distinct valencies.
Moreover, in the case where there are vertices $y$ and $z$ outside of $C$ having two distinct valencies, we have $\alpha_x = \alpha_y + \alpha_z$.
Hence the valency of any other vertex of $\Gamma$ must be $d_x$, $d_y$, or $d_z$, and the claim is established.
One can also see that, since for each vertex $v$ each $\alpha_v$ is positive, $x$ has the largest of the valencies.
Therefore we have also established the following claim (the claim is trivial if $\Gamma$ has only two valencies).

\paragraph{\textbf{Claim 2}} The vertices of the largest (or larger) valency in $\Gamma$ induce a regular connected component of $\overline \Gamma$. 
\label{par:claim2}

If $y$ and $z$ were in different connected components, then we could simultaneously write $\alpha_x = \alpha_y + \alpha_z$, $\alpha_y = \alpha_z + \alpha_x$, and $\alpha_z = \alpha_x + \alpha_y$ which is clearly impossible.
Thus we can also deduce the following.

\paragraph{\textbf{Claim 3}} $\overline \Gamma$ has precisely two connected components. 





Let $C$ be a regular connected component of $\overline{\Gamma}$ and let $\bf{v}$ be an eigenvector of $C$ with non-trivial eigenvalue $\theta$ (i.e., $\bf{v}$ is orthogonal to the `all ones' vector).
Let ${\bf w} \in {\mathbb R}^V$ be defined by $w_x = v_x$ if $x \in V(C)$ and $0$ otherwise.
Then ${\bf w}$ is an eigenvector of $\Gamma$ with eigenvalue $-\theta-1$. 
This means that $C$ has at most three distinct eigenvalues and is either a complete graph $K_t$ with $t \geqslant 1$ or a strongly regular graph.
First suppose $C = K_t$.
If $t=1$ then $\Gamma$ has a vertex of valency $n-1$, i.e., $\Gamma$ is a cone.
Otherwise $t \geqslant 2$, in which case $\Gamma$ has an eigenvalue $0$.
Hence $\Gamma$ is complete multipartite and, since $\overline{\Gamma}$ has only two connected components, $\Gamma$ must be complete bipartite.

It remains to consider the case when each regular connected component of $\overline{\Gamma}$ is a strongly regular graph.
Let $k_1 > k_2$ be the two largest valencies of $\Gamma$ and let $\regSub$ be the regular subgraph of $\Gamma$ induced on $V_1$ with $n_1$ vertices and valency $k_{11}$.
By above, we assume that $\regSub$ is strongly regular and, by interlacing, it must have eigenvalues, $k_{11}$, $\theta_1$, and $\theta_2$.
Let vertices $x$ and $y$ have respective valencies $d_x = k_1$ and $d_y = k_2$.
Note that $x \sim y$ since they must be in different connected components of $\overline \Gamma$.
Hence we have 
\begin{equation}
	\label{eqn:comps1}
	k_2 - \nu_{x,y} - 1 = \alpha_2^2 - \theta_1 \theta_2 - \alpha_1\alpha_2 - \theta_1 -\theta_2 -1=
	 \alpha_2 (\alpha_2 - \alpha_1) - (\theta_1 +1)(\theta_2 + 1).
\end{equation}
On the other hand, since the complement $\overline \regSub$ of $\regSub$ is strongly regular with valency $n_1 - k_{11} - 1$ and non-trivial eigenvalues $-\theta_1-1$ and $-\theta_2-1$, we can also write
\begin{equation}
	\label{eqn:comps2}
	k_2 - \nu_{x,y} -1 = n_1 - k_{11} -1 = \overline{\mu} - (\theta_1 + 1)(\theta_2+1),
\end{equation}
where $\overline{\mu}$ is the number of common neighbours of two non-adjacent vertices in the component $\overline{\regSub}$ of $\overline{\Gamma}$. 
The second equality follows from a well-known equality~\cite[Thm 1.3.1]{bcn89}.
By comparing Eqs.~\eqref{eqn:comps1} and \eqref{eqn:comps2}, since $\alpha_1 > \alpha_2$, we have $\overline{\mu} < 0$, which is impossible.
\end{proof}

Note that if a bipartite graph is not complete bipartite then its diameter must be at least $3$ and hence it cannot have fewer than $4$ distinct eigenvalues.
The next result follows from this observation and from the above theorem, which shows that the complete bipartite graphs are the only complete multipartite graphs with precisely three distinct eigenvalues. 

\begin{corollary}\label{cor:compbi}
	Let $\Gamma$ be a graph in $\mathcal G(\theta_0,\theta_1,\theta_2)$.
	Then the following are equivalent.
	\begin{enumerate}[(i)]
		\item $\Gamma$ is bipartite;
		\item $\Gamma$ is complete bipartite;
		\item $\theta_1 = 0$.
	\end{enumerate}
\end{corollary}

\begin{remark}
	\label{rem:alphaints}
	We remark that if $\Gamma$ is a non-bipartite graph in $\mathcal G(\theta_0,\theta_1,\theta_2)$ then, by Theorem~\ref{thm:integralsr}, $\theta_0$ is an integer.
	It follows that both $\theta_1+\theta_2$ and $\theta_1\theta_2$ are also integers.
	Hence, for all vertices $x, y \in V(\Gamma)$, it is evident from the equations for $d_x$ and $\nu_{x,y}$ that $\alpha_x^2$ and $\alpha_x\alpha_y$ are integers.
\end{remark}

In the proof of Theorem~\ref{thm:disconn} we also saw that the disconnected complement of $\Gamma$ must have at most three valencies.
Hence we have the following corollary.

\begin{corollary}[cf. \cite{Dam3ev}]\label{cor:cone3}
	Let $\Gamma$ be a cone in $\mathcal G(\theta_0,\theta_1,\theta_2)$.
	Then $\Gamma$ has at most three valencies.
\end{corollary}

\begin{remark}
	The above corollary generalises a result of Bridges and Mena~\cite{brme} who studied cones having distinct eigenvalues $\theta_0$, $\theta_1$, and $-\theta_1$.
	They proved that, except for at most three cones having three valencies, such graphs are cones over strongly regular graphs with parameters $(\lambda^{3}+2\lambda^{2}, \lambda^{2}+\lambda, \lambda,\lambda)$.
	(Only two of these three exceptional cones have been constructed, it is still an open problem to decide the existence of the largest cone.)
\end{remark}
%

\section{Bounds for graphs with three eigenvalues} 
\label{sec:bounds_for_graphs_with_three_eigenvalues}

In this section we give a series of bounds for graphs that have precisely three distinct eigenvalues.
These include bounds for the valencies, bounds for the eigenvalues, and bounds for the entries of the Perron-Frobenius eigenvector $\alpha$. 

\begin{lemma}\label{lem:bound1} 
	Let $\Gamma$ be a non-bipartite graph in $\mathcal G(\theta_0,\theta_1,\theta_2)$ and let $x$ and $y$ be vertices with respective valencies $d_x > d_y$.
	Then the following hold:
	\begin{enumerate}[(i)]
		\item if $x \sim y$, then  $\alpha_x -1 \leqslant (\alpha_x - \alpha_y)\alpha_y \leqslant -(\theta_1 +1)(\theta_2 +1)$ and $\alpha_x \alpha_y \geqslant -\theta_1 - \theta_2$;
		\item if $x \not \sim y$, then  $\alpha_x -1 \leqslant (\alpha_x - \alpha_y)\alpha_y \leqslant -\theta_1 \theta_2 $.
	\end{enumerate}
\end{lemma}

\begin{proof}
	Since $\alpha_x^2$, $\alpha_{x}\alpha_y$, and $\alpha_y^2$ are all integers (see Remark~\ref{rem:alphaints}) and $\alpha_x > \alpha_y \geqslant 1$, we have $\alpha_x \geqslant \alpha_y + 1$ and hence $\alpha_x - 1 \leqslant (\alpha_x - \alpha_y)\alpha_y$.
	The rest follows from the fact that $0 \leqslant \nu_{x,y} \leqslant d_y - 1$ when $x \sim y$ and $\nu_{x,y} \leqslant d_y$ when $x \not \sim y$.
\end{proof}

Van Dam and Kooij~\cite{DK} showed that the number $n$ of vertices of a connected graph $\Gamma$ with diameter 2 with spectral radius $\rho$ satisfies $n \leqslant \rho^2 +1$ with equality if and only if $\Gamma$ is a Moore graph of diameter 2 or $\Gamma$ is the $K_{1, n-1}$. 
As a consequence we have the following lemma.

\begin{lemma}\label{lem:rhobound}
	Let $\Gamma$ be an $n$-vertex graph in $\mathcal G(\theta_0,\theta_1,\theta_2)$.
	Let $k_{\text{min}}$, $k_{\text{avg}}$, and $k_{\text{max}}$ respectively denote the smallest, average, and largest valency of $\Gamma$.
	Then $k_{\text{min}} < k_{\text{avg}} < \theta_0 < k_{\text{max}}$ and $n \leqslant \theta_0^2 + 1$ with equality if and only if $\Gamma$ is $K_{1,n-1}$.
\end{lemma}

Now we can establish bounds on the size of the largest valency and the number of vertices.

\begin{proposition}\label{pro:bound2}
	Let $\Gamma$ be a non-bipartite graph in $\mathcal G(\theta_0,\theta_1,\theta_2)$. 
	Let $k_{\text{max}}$ be the maximal valency in $\Gamma$ and let $\ell := \min\{ 1-(\theta_1 +1)(\theta_2 +1), -\theta_1 \theta_2 +1\}$.  
	Then the following hold:
	\begin{enumerate}[(i)]
		\item $k_{\text{max}} \leqslant (1-(\theta_1 +1)(\theta_2 +1))^2 - \theta_1 \theta_2$ ;
		\item if the complement of $\Gamma$ is connected, then $k_{\text{max}} \leqslant \ell^2 - \theta_1 \theta_2 $;
		\item $n \leqslant \max\{ (\ell^2-\theta_1\theta_2 -1)^2 +1,  (1-(\theta_1 +1)(\theta_2 +1))^2 - \theta_1 \theta_2 +1\} $. 
	\end{enumerate}
\end{proposition}

\begin{proof} 
	Let $x$ be a vertex with valency $k_{\text{max}}$, having a neighbour $y$ with $d_y < k_{\text{max}}$. 
	Now (i) follows from Lemma \ref{lem:bound1}.
	If all vertices $x$ with valency $k_{\text{max}}$ do not have a non-neighbour $y$ with $d_y < k_{\text{max}}$, then the complement of $\Gamma$ is not connected.
	So we may assume that there is a vertex $x$ with valency $k_{\text{max}}$ having a non-neighbour $y$ with $d_y < k_{\text{max}}$. 
	Hence (ii) follows from Lemma \ref{lem:bound1}.
	For (iii), we have by Lemma \ref{lem:rhobound} that $n \leqslant \theta_0^2 +1 \leqslant (k_{\text{max}}-1)^2 +1$.
	So the result follows by Theorem~\ref{thm:disconn} and (ii) if $k_{\text{max}} \neq n-1$. 
	Otherwise, if $k_{\text{max}} = n-1$ then $n =  k_{\text{max}} +1 \leqslant (1-(\theta_1 +1)(\theta_2 +1))^2 - \theta_1 \theta_2 +1$.
\end{proof}

The proof of the next result is essentially the same as the proof of Proposition~3 in \cite{Dam3ev}.

\begin{lemma}\label{lem:vanDamprop3}
	Let $\Gamma$ be an $n$-vertex non-bipartite graph in $\mathcal G(\theta_0,\theta_1,\theta_2)$.
	Assume that $\theta_1$ and $\theta_2$ have the same multiplicity, $m =(n-1)/2$.
	Then there exists a positive integer $t$ such that $\theta_1 = (-1+\sqrt{4t+1})/2$, $\theta_2= (-1 - \sqrt{4t+1})/2$, and $\theta_0 = (n-1)/2$ and $4t +3 \leqslant n$.
\end{lemma}

Van Dam's \cite[Proposition~3]{Dam3ev} assumes that $\theta_1$ and $\theta_2$ are irrational and his proof uses that fact that $\theta_1$ and $\theta_2$ have the same multiplicity.
Our Lemma~\ref{lem:vanDamprop3} merely starts with assuming that $\theta_1$ and $\theta_2$ have the same multiplicity, hence the proof follows in the same way.

\begin{proposition}\label{nonintegral}
	Let $\Gamma$ be an $n$-vertex non-bipartite graph in $\mathcal G(\theta_0,\theta_1,\theta_2)$.
	Assume that $\theta_1$ and $\theta_2$ have the same multiplicity.
	Let $t$ be defined as in Lemma~\ref{lem:vanDamprop3}.
	Then the following hold:
	\begin{enumerate}[(i)]
		\item For the maximal valency $k_{\text{max}}$ in $\Gamma$ we have $k_{\text{max}} \leqslant t^2 + 3t+ 1$;
		\item For the number of vertices we have $4t +3 \leqslant n \leqslant 2 k_{\text{max}} -1 \leqslant 2t^2 + 6t +1 $.
	\end{enumerate}
\end{proposition}

\begin{proof}
	Let $x$ be a vertex with valency $k_{\text{max}}$.
For the first part, we have $\alpha_x = \sqrt{d_x + \theta_1 \theta_2} = \sqrt{d_x - t}$.
Then $\alpha_x \leqslant -(\theta_1 +1)( \theta_2 +1) +1 = t +1$.

For the second part, we have $k_{\text{max}} \geqslant (n+1)/2$ since $ k_{\text{max}} > \theta_0 = (n-1)/2$ is an integer. 
The upper bound follows easily. 
The lower bound follows from Lemma~\ref{lem:vanDamprop3}.
\end{proof}

To prove our next result we will need a theorem of Bell and Rowlinson which enables us to bound the number of vertices of a graph in terms of the multiplicity of one of its eigenvalues.

\begin{theorem}[See \cite{br03}]\label{thm:bell} 
	Let $\Gamma$ be a graph on $n$ vertices with an eigenvalue $\theta$ with multiplicity $n-t$ for some positive integer $t$.
	Then either $\theta \in \{0, -1\}$ or $n \leqslant \frac{t(t+1)}{2}$.
\end{theorem}

We note that, using the classification of Van Dam~\cite[Section 7]{Dam3ev} of graphs with three distinct eigenvalues with at most $29$ vertices and the above result of Bell and Rowlinson, we obtain readily the graphs in $\mathcal G(\theta_0,\theta_1,\theta_2)$ where the multiplicity of $\theta_1$ or $\theta_2$ is at most $6$.

\begin{lemma}\label{lem:th2bound}
	Let $\Gamma$ be an $n$-vertex graph in $\mathcal G(\theta_0,\theta_1,\theta_2)$.
	For $\{ \theta_l, \theta_s \} = \{ \theta_1, \theta_2 \}$ where the multiplicity of $\theta_l$ is at least that of $\theta_s$, we have the following inequalities.
	\begin{align*}
		\theta_l^2 &\leqslant 2(n-(1-1/n));  \\ 
		\theta_s^2 &\leqslant n \sqrt{(n-1)/2} + 1/(2(n-1)).
	\end{align*}
\end{lemma}
\begin{proof}
	Lemma~\ref{lem:vanDamprop3} deals with the case when the multiplicities are equal.
	We can therefore assume that the multiplicity $m_l$ of $\theta_l$ is at least $n/2$.
	Let $A$ be the adjacency matrix of $\Gamma$ and let $k_1$ be the largest valency of the vertices of $\Gamma$.
	Then since $n k_1 \geqslant \operatorname{tr}(A^2)$ we have
	\[
		n^2 \geqslant n k_1 \geqslant \theta_0^2 + m_1 \theta_1^2 + m_2 \theta_2^2 \geqslant n-1 + n/2 \theta_l^2.
	\] 
	This gives the first inequality.
	By Theorem~\ref{thm:bell}, the multiplicity $m_s$ of $\theta_s$ is at least $\sqrt{2(n-1)}$.
	Then the second inequality follows in the same way as above.
\end{proof}

Lemma~\ref{lem:th2bound} gives us a crude bound on the size of the two smaller eigenvalues of a graph with precisely three distinct eigenvalues.


	We call the $11$-vertex cone over the Petersen graph, the \textbf{Petersen cone} (see \cite[Fig. 1]{Dam3ev}) and the \textbf{Van Dam-Fano graph} (see \cite[Fig. 2]{Dam3ev}), the graph formed by taking the incidence graph of the complement of the Fano plane, and adding edges between every pair of blocks.

\begin{lemma}\label{lem:th0bound}
	Let $\Gamma$ be an $n$-vertex non-bipartite graph in $\mathcal G(\theta_0,\theta_1,\theta_2)$.
	Then $\theta_0 \leqslant n - 6$ with equality if and only if $\Gamma$ is the Petersen cone or the Van Dam-Fano graph.
\end{lemma}

\begin{proof}
	By Theorem~\ref{thm:integralsr} we know that $\theta_0$ is rational.
	First suppose that $\theta_1$ and $\theta_2$ are irrational.
	Then they must have the same multiplicities, hence we can apply Lemma~\ref{lem:vanDamprop3} to obtain the equality $\theta_0 = (n-1)/2$.
	If $n - 6 \leqslant (n-1)/2$ then $n \leqslant 11$, but no such graph exists (all graphs on up to $29$ vertices have been classified by Van Dam).
	
	Now suppose all eigenvalues are rational.
	If $\theta_2 = -2$ then the lemma holds by examination of the classification theorem of Van Dam~\cite[Theorem 7]{Dam3ev}.
	It remains to assume $\theta_2 \leqslant -3$ and we assume that $n-6 \leqslant \theta_0 \leqslant n-2$ and $n \geqslant 30$.
	Then, by Lemma~\ref{lem:rhobound}, we can write $\theta_0^2 + m_1 \theta_1^2 + m_2 \theta_2^2 = \sum_{v \in \Gamma} d_v < n \theta_0$. 
	Since $m_1 +m_2 = n-1$ and $\theta_0(n-\theta_0) \leqslant 6(n-6)$, we have $\min\{\theta_1, -\theta_2\} < \sqrt{6}$.
	Hence $\theta_1 \leqslant 2$.

	Together with the expressions for the multiplicities $m_1$ and $m_2$ from Eq. \eqref{mult}, we deduce that 
	\begin{equation}
		\label{eqn:inequ1}
		3 \leqslant -\theta_2 < \frac{\theta_0(n - \theta_0 + \theta_1)}{\theta_0 + (n-1)\theta_1}.
	\end{equation}
		
    If $\theta_1 =2$, then we have $ -\theta_2 < 8 (n-6)/(3n - 8) < 3$; a contradiction.
    It remains to consider $\theta_1 = 1$.
	Here we have $ -\theta_2 \leqslant 7(n-6)/(2n - 7) < 4$, hence $\theta_2 = -3$.
	Using the multiplicity equations \eqref{mult}, we can write the multiplicity of $\theta_2$ as $m_2=(n-1+\theta_0)/4$.
	Since $m_2$ is an integer, $\theta_0$ must be either $n-5$ or $n-3$.
	But then, in either case, \eqref{eqn:inequ1} gives $-\theta_2 < 3$, which is a contradiction.
\end{proof}


\section{Biregular graphs with three eigenvalues}
\label{sec:bireg}

In this section we focus on biregular graphs with precisely three distinct eigenvalues.

\subsection{Computing feasible valency-arrays and spectra} 
\label{sub:computing_feasible_parameters}

Let $\Gamma$ be a graph having $r$ distinct valencies $k_1>\dots>k_r$ with multiplicities $n_1,\dots,n_r$, i.e., $n_i := |\{v \in V(\Gamma) : d_v = k_i \}|$.
The \textbf{valency-array} of $\Gamma$ is defined to be the tuple $(n_1,\dots,n_r ; k_1, \dots, k_r)$.
The main result (Theorem~\ref{prop2val}) of this section gives us strong restrictions on the valency-arrays of biregular graphs with three eigenvalues.

We begin with a result about biregular cones.  

\begin{proposition}[See \cite{Dam3ev}]\label{pro:biregcones}
	Let $\Gamma$ be a biregular cone in $\mathcal G(\theta_0,\theta_1,\theta_2)$.
	Then $\Gamma$ is a cone over a strongly regular graph.
\end{proposition}

\begin{remark}
	\label{rem:WL}
	Our Theorem~\ref{thm:disconn} is reminiscent of Proposition 6.1 (a) given by Muzychuk and Klin~\cite{MK}.
	Let $\Gamma \in \mathcal G(\theta_0,\theta_1,\theta_2)$ and let $W(\Gamma)$ denote its \emph{Weisfeiler-Lehman closure} (see \cite[Section 6]{MK}).
	We also remark that, by Theorem~\ref{thm:disconn} and Proposition~\ref{pro:biregcones}, we see that \cite[Proposition 6.1 (a)]{MK} says that if $\dim(W(\Gamma)) = 6$ then $\Gamma$ is biregular with a disconnected complement.
	Muzychuk and Klin~\cite{MK} suggest classifying all graphs $\Gamma \in \mathcal G(\theta_0,\theta_1,\theta_2)$ satisfying $\dim(W(\Gamma)) = 9$, which is the next interesting case after $\dim(W(\Gamma)) = 6$.
\end{remark}
Van Dam~\cite{Dam3ev} showed that if a graph $\Gamma$ has precisely three distinct eigenvalues and at most three distinct valencies then the valency partition is equitable.
We show a slightly refined version of this result where we assume that $\Gamma$ has precisely two distinct valencies.

\begin{theorem}\label{prop2val}
	Let $\Gamma$ be an $n$-vertex non-bipartite biregular graph in $\mathcal G(\theta_0,\theta_1,\theta_2)$ with valency-array $(n_1,n_2;k_1,k_2)$. 
	Then the following conditions hold:
	\begin{enumerate}[(i)]
		\item The valency partition $\{ V_1, V_2\}$ is an equitable partition of $\Gamma$ with quotient matrix $Q = \begin{pmatrix} k_{11} & k_{12} \\ k_{21} & k_{22} \end{pmatrix}$, where
	$$ k_{11} = \frac{\alpha_1 \theta_0 - \alpha_2 k_1}{\alpha_1 - \alpha_2}, \;
	   k_{12} = \alpha_1\frac{k_1 - \theta_0}{\alpha_1 - \alpha_2}, \;
	    k_{21} = \alpha_2 \frac{\theta_0 - k_2}{\alpha_1 - \alpha_2}, \;
		 k_{22} =\frac{\alpha_1 k_2 - \alpha_2 \theta_0}{\alpha_1 - \alpha_2}.$$
		\item All eigenvalues of $\Gamma$ are integers.
		\item If the matrix $Q$ has eigenvalues $\theta_0 $ and $\theta$, then $\alpha_1 \alpha_2 = -\theta(\theta' +1)$ where $\{ \theta, \theta'\} = \{ \theta_1, \theta_2\}$. In particular, if $k_{11} =0$ or $k_{22} =0$ then $\alpha_1 \alpha_2 =- \theta_2(\theta_1 +1)$.
		\item We have $$n = \frac{(\alpha_1^2 + \alpha_1 \alpha_2 + \alpha_2^2 - \theta_0 -\theta_1 \theta_2)(\theta_0 - \theta_1)(\theta_0 - \theta_2)}{(\theta_0 + \theta_1 \theta_2 + \alpha_1 \alpha_2)\alpha_1 \alpha_2};$$
	$$n_1 =  \frac{(\theta_0 - \alpha_2^2 + \theta_1 \theta_2)(\theta_0 - \theta_1)(\theta_0 - \theta_2)}{(\theta_0 + \theta_1 \theta_2 + \alpha_1 \alpha_2)(\alpha_1 - \alpha_2)\alpha_1 };$$
	 $$n_2 =  \frac{(\alpha_1^2 - \theta_0 - \theta_1 \theta_2)(\theta_0 - \theta_1)(\theta_0 - \theta_2)}{(\theta_0 + \theta_1 \theta_2 + \alpha_1 \alpha_2)(\alpha_1 - \alpha_2)\alpha_2 }.$$
		\item The following conditions are equivalent: 
		\begin{enumerate}
			\item $k_{21} = n_1$;
			\item $k_{12} = n_2$;
			\item $n_1 = 1$;
			\item $\Gamma$ is a cone over a strongly regular graph.
		\end{enumerate}
		\item If $n$ is a prime at least $3$, then $\Gamma$ is a cone over a strongly regular graph.
		\item  $\alpha_1 -1 \leqslant (\alpha_1 - \alpha_2)\alpha_2 \leqslant \min \{-(\theta_1 +1)(\theta_2 +1), -\theta_1\theta_2\} $, unless $\Gamma$ is a cone over a strongly regular graph.
	\end{enumerate}
\end{theorem}
\begin{proof}
We will prove each part of the theorem in turn.
\begin{enumerate}[(i)]
	\item Let a vertex $x$ of valency $k_1$ have $k_{11}$ neighbours in $V_1$ and $k_{12} := k_1 - k_{11} $ neighbours in $V_2$.
The vector $\alpha$ is the $\theta_0$-eigenvector of $\Gamma$, therefore $k_{11} \alpha_1 + k_{12} \alpha_2 = \theta_0 \alpha_1$. 
Since $\alpha_1 > \alpha_2$, it follows that
\[
	k_{11} = \frac{\alpha_1 \theta_0 - \alpha_2 k_1}{\alpha_1 - \alpha_2} \quad \text{ and } \quad k_{12} = \alpha_1 \frac{k_1 -  \theta_0}{\alpha_1 - \alpha_2}.
\]
Applying this idea again to a vertex of valency $k_2$ gives the first part of the theorem.
	
	\item By Theorem~\ref{thm:integralsr}, $\theta_0$ is an integer.
	Since $\theta_0$ and $\theta$ are eigenvalues of $Q$, we have $\theta_0 + \theta = k_{11} + k_{22} \in \Z$ and hence $\theta \in \Z$.
	The trace of the adjacency matrix of $\Gamma$ is zero, whence the remaining eigenvalue of $\Gamma$ is integral.
	
	\item Clearly $Q$ has $\theta_0$ as an eigenvalue (with eigenvector $(\alpha_1, \alpha_2)$).
The eigenvalues of $Q$ are a subset (with multiplicity) of the eigenvalues of $\Gamma$, hence the other eigenvalue $\theta$ of $Q$ is in $\{\theta_1, \theta_2\}$.
Note that if $k_{11}= 0$ or $k_{22} =0$ then the determinant of $Q$ is negative, hence $Q$ has a negative eigenvalue, namely, $\theta_2$.
	Taking the determinant of $Q$, $\det Q = k_{11} k_{22} - k_{12}k_{21}$ and using the expressions for the $k_{ij}$ in (i), one obtains the expression for $\alpha_1\alpha_2$.

	\item Since the valency partition is equitable we have $k_{12} n_1 = k_{21} n_2$.
	Moreover, from $n_1 + n_2 = n$, we obtain that
\begin{equation}
	\label{n1n2} 
	n_1 = \frac{k_{21}}{k_{12} + k_{21}} n \quad \text{ and } \quad  n_2 = \frac{k_{12}}{k_{12} + k_{21}} n.
\end{equation}
From, $\theta_0^2 + m_1 \theta_1^2 + m_2 \theta_2^2 = n_1 k_1 + n_2 k_2$, using the multiplicity equations \eqref{mult} and the formulae for the $k_{ij}$'s in (i), one readily obtains the formula for $n$. 
The formulae for $n_1$ and $n_2$ follow easily.

	\item 
%
	
	That (a) is equivalent to (b) follows from (iv).
	Both (a) and (b) imply that the complement of $\Gamma$ is disconnected, which implies (c) by Theorem~\ref{thm:disconn}.
	By Proposition~\ref{pro:biregcones} (c) implies (d) and clearly (d) implies (a).

	\item If $n$ is a prime then by Eqs. \eqref{n1n2} one sees that $n_1$ and $n_2$ must be equal to $k_{21}$ and $k_{12}$. 
	Then use (v).
	
	\item This follows from Proposition~\ref{pro:bound2}.
\end{enumerate}  \vspace{-0.5cm} 
\end{proof}

We call a valency-array and spectra \textbf{feasible} if they together satisfy the assumptions of Theorem~\ref{prop2val} such that the matrix $Q$ is a nonnegative integer matrix and the quantities $n_1$, $n_2$, $m_1$, and $m_2$ are positive integers.
Using Theorem~\ref{prop2val} and Lemma~\ref{lem:th2bound} we have compiled a table of the feasible valency-arrays and spectra for biregular graphs with precisely three distinct eigenvalues, see the appendix.
\begin{remark}
	\label{rem:pairVA}
	Let $\Gamma$ be a biregular graph in $\mathcal G(\theta_0,\theta_1,\theta_2)$ with the spectrum of $\Gamma$ fixed.
	Using Theorem~\ref{prop2val} one can see that $\Gamma$ can have at most two possible valency-arrays.
	The problem of determining if the valency-array of $\Gamma$ is determined by its spectrum comes down to Diophantine analysis.
	So far we do not have any examples of a pair of valency-arrays corresponding to graphs with the same spectrum.
\end{remark}


\subsection{Bounding the second largest eigenvalue} 
\label{sub:bounding_the_second_largest_eigenvalue}

Neumaier~\cite{Neu} showed that for a fixed $m$, all but finitely many primitive strongly regular graphs with smallest eigenvalue at least $-m$ fall into two infinite families.

\begin{theorem}[See~\cite{Neu}]\label{thm:neuthm}
Let $m \geqslant 1$ be a fixed integer. 
Then there exists a constant $C(m)$ such that any connected and coconnected strongly regular graph $\Gamma$ with smallest eigenvalue $-m$ having more than $C(m)$ vertices has the following parameters (given in the form $\SRG(n, k, \lambda, \mu )$).
	\begin{enumerate}[(i)]
		\item $\SRG(n, sm, s-1 + (m-1)^2, m^2 )$ where $s \in \mathbb{N}$ and $n = (s+1)(s(m-1) +m)/m$;  or
		\item $\SRG((s+1)^2, sm, s-1+(m -2)(m-1), m(m-1))$ where $s \in \mathbb{N}$.
	\end{enumerate}
\end{theorem}

In the next result we show that for fixed $\theta \neq 0$ there are only finitely many cones over strongly regular graphs with exactly three distinct eigenvalues and one of them equal to $\theta$.

\begin{lemma} \label{conesrg}
	Let $\theta \neq 0, -1$ be a fixed algebraic integer
	and let $\Gamma \in \mathcal G(\theta_0,\theta_1,\theta_2)$ be a cone over a strongly regular graph $\regSub$ with $\theta \in \{ \theta_0, \theta_1, \theta_2\}$.
	Then $\Gamma$ is one of a finite number of graphs.
\end{lemma}
\begin{proof} 
	If $\theta= \theta_0$, then the result follows from Lemma \ref{lem:rhobound}.
	
	Suppose that $\Gamma$ has at least $4$ vertices.
	By Theorem~\ref{thm:bell}, the multiplicities of both of the eigenvalues $\theta_1$ and $\theta_2$ are at least $2$.
	Hence, by interlacing, $\regSub$ has eigenvalues $k$, $\theta_1$, and $\theta_2$, for some $k$.
	
	First suppose $\theta = \theta_2$.
	By Theorem \ref{thm:neuthm}, there exists a constant $C(-\theta)$ such that $n \leqslant C(-\theta)$, otherwise, for a positive integer $s$, either $\regSub$ has parameters $\SRG(n, s(-\theta), s-1 + (-\theta-1)^2, \theta^2 )$ and $n = (s+1)(s(-\theta-1)-\theta)/(-\theta)$ or $\regSub$ has parameters
$\SRG((s+1)^2, -s\theta, s-1+(\theta +2)(\theta+1), \theta(\theta+1))$. 
	Since $\Gamma$ is the cone over $\regSub$, by Proposition~\ref{propcone}, the equation $\theta(k- \theta) = -n$ must be satisfied.
	It is easily checked that $s$ must equal either $\theta^2 - \theta$ or $\theta^2-1$ and hence $\regSub$ can be one of only finitely many graphs.

	Finally, suppose instead that the second largest eigenvalue of $\regSub$ is $\theta = \theta_1$.
	Since the complement of a strongly regular graph is also a strongly regular graph, we can apply the same argument as above where we consider $\regSub$ to be the complement of a strongly regular graph.
\end{proof}

Now we show that, there are only finitely many connected biregular graphs with three distinct eigenvalues and bounded second largest eigenvalue.

\begin{theorem}
	\label{thm:boundsec}
	Let $\Gamma$ be an $n$-vertex biregular graph in $\mathcal G(\theta_0,\theta_1,\theta_2)$ with valencies $k_1 > k_2$ and let $t$ be a positive integer.
	Then there exists a constant $C(t)$ such that if $0 < \theta_1 \leqslant t$, then  $n \leqslant C(t)$.
\end{theorem}
\begin{proof} 
 	If $\theta_2 \geqslant -2t$, then the existence of $C(t)$ follows from Proposition~\ref{pro:bound2}.
	So from now on we will assume $\theta_2 < -2t$.

	From Theorem~\ref{prop2val}, the eigenvalue $\theta_1$ is an integer and without loss of generality we may assume that $\theta_1 = t$.
	If $\Gamma$ is a cone over a connected strongly regular graph, then the existence of $C(t)$ follows from Lemma \ref{conesrg}, hence from now on we assume that $\Gamma$ is not a cone over a strongly regular graph.

	By Theorem~\ref{prop2val} we have two cases to consider, namely the case  $\alpha_1 \alpha_2 = -\theta_2(\theta_1 +1)$ and the case $\alpha_1 \alpha_2 = -\theta_1(\theta_2 +1)$.
	Let us first consider the case $\alpha_1 \alpha_2 = -\theta_2(\theta_1 +1)$.
	Since $\Gamma$ is not a cone, there exist vertices $x$ and $y$ with respective valencies $k_1$ and $k_2$ such that $x \not \sim y$.
	Therefore $\nu_{x,y} = \alpha_1\alpha_2 \leqslant k_2 = \alpha_2^2- \theta_1\theta_2$.
	Using $\alpha_1 \alpha_2  = -\theta_2(\theta_1 +1)$, we see that $\alpha_2^2 \geqslant -\theta_2$ and hence $\alpha_1^2 \leqslant -(\theta_1+1)^2\theta_2 = -(t+1)^2 \theta_2$.
	Thus any two non-adjacent vertices have at least $\alpha_2^2 \geqslant -\theta_2 $ common neighbours.
	Hence, since $\Gamma$ has diameter $2$,
	$$ 
	n \leqslant 1 + k_2 + k_2(k_1-1)/\alpha_2^2 = 1 + k_2 + k_1 -1 -t \theta_2(k_1-1)/\alpha_2^2 < (2  + t)k_1.
	$$
	We also have
	$ k_1 = \alpha_1^2 - \theta_1 \theta_2 \leqslant -\theta_2(t^2 + 3t +1)$.

	Now we find that for the multiplicity $m_2$ the following holds:
	\[
		m_2 = \frac{(n-1)\theta_1 + \theta_0} {\theta_1 - \theta_2} < \frac{(2+t)k_1t + k_1}{-\theta_2} = \frac{(t+1)^2k_1}{-\theta_2}  \leqslant (t+1)^2(t^2 + 3t +1).
	\]
	Then Theorem~\ref{thm:bell} yields $n \leqslant  ((t+1)^2(t^2 + 3t +1)+1 )((t+1)^2(t^2 + 3t +1)+2)/2$. 
	This shows the existence of $C(t)$ in this case.

	It remains to consider the case $\alpha_1 \alpha_2  = -\theta_1(\theta_2 +1)$.
	If $k_{22} = 0$ then, by Theorem~\ref{prop2val}, $\alpha_1 \alpha_2 = -\theta_2(\theta_1 +1)$ which we have dealt with above.
	We therefore can assume that there exist $x,y \in V_2$ with $x \sim y$.
 	Hence $\nu_{x,y} = \alpha_2^2 + \theta_1 + \theta_2 \geqslant 0$.
	We conclude that $\alpha_2^2 \geqslant -\theta_2/2$.
	Now the bound follows in a fashion similar to the case $\alpha_1 \alpha_2  = -\theta_2(\theta_1 +1)$.
\end{proof}

Note that the above result is not true for connected graphs with exactly $4$ distinct eigenvalues and exactly two distinct valencies.
Indeed, the friendship graphs, i.e., cones over a disjoint union of copies of $K_2$, can have unbounded number of vertices and all but two of the eigenvalues are equal to $\pm 1$.

\subsection{Second largest eigenvalue 1}
\label{sec:sec1}
In this section we will determine the connected biregular graphs with three distinct eigenvalues and second largest eigenvalue $1$.
First we determine the cones of strongly regular graphs with second largest eigenvalue $1$.
Seidel~\cite{sei} (see also \cite[Thm 3.12.4 (i)]{bcn89}) classified the strongly regular graphs with smallest eigenvalue $-2$.

\begin{theorem}[\cite{sei}]
	\label{thm:Seidel}
	Let $\Gamma$ be a connected strongly regular graph with smallest eigenvalue $-2$.
	Then $\Gamma$ is either
		 a triangular graph $T(m)$ for $m \geqslant 5$;
		 an $(m \times m)$-grid for $m \geqslant 3$;
		 the Petersen graph;
		 the Shrikhande graph;
		 the Clebsch graph;
		 the Schl\"{a}fli graph;
		 or one of the three Chang graphs.

\end{theorem}

\begin{lemma}\label{Petersen} Let $\Gamma$ be a cone over a strongly regular graph in $\mathcal G(\theta_0,1,\theta_2)$.
	Then $\Gamma$ is the Petersen cone.
\end{lemma}
\begin{proof} The strongly regular graphs with second largest eigenvalue $1$ are exactly the complements of the graphs in Theorem~\ref{thm:Seidel}.
	Checking whether each graph satisfies the condition of Proposition~\ref{propcone} gives the lemma.
\end{proof}

\begin{lemma}\label{lem:paramnonexi}
	There do not exist graphs with the following:
	\begin{enumerate}[(i)]
		\item valency-array $(15,15;20,8)$ and spectrum $\{ [16]^1, [1]^{20}, [-4]^9 \}$;
		\item valency-array $(12,39;35,14)$ and spectrum $\{ [22]^1, [1]^{41}, [-7]^9 \}$.
	\end{enumerate}
\end{lemma}
\begin{proof}
	The valency partitions corresponding to the graphs satisfying (i) and (ii) have quotient matrices
	\[
		\begin{pmatrix}
				12 & 8 \\
				8 & 0
		\end{pmatrix}
		 \quad \text{ and } \quad
		\begin{pmatrix}
			9 & 26 \\
			8 & 6
			\end{pmatrix},
	\]
	respectively.
	The first case was shown to be impossible by Van Dam~\cite{Dam3ev}.
	In the second case, for any two vertices $x \sim y$ in $V_2$, we have $\nu_{x,y} =1$, but this is impossible as both are adjacent to $8$ out of the $12$ vertices in $V_1$.
\end{proof}

\begin{prop}\label{secondev1}
	Let $\Gamma$ be a biregular graph in $\mathcal G(\theta_0,1,-t)$.
	Then $t=2$, and $\Gamma$ is the Petersen cone or Van Dam-Fano graph.
\end{prop}
\begin{proof}
	Suppose $\Gamma$ has valencies $k_1 > k_2$.
	By Theorem \ref{prop2val}, we have $\alpha_1 \alpha_2=t-1$ or $\alpha_1 \alpha_2=2t$.
Let us first consider the case $\alpha_1 \alpha_2=t-1$.
Let $v,w \in V_2$ and $v\sim w$.
Then we have $\nu_{v,w}=1-t+\alpha_2^2<1-t+\alpha_1\alpha_2=0$.
Hence we must have $k_{22}=0$, but then, by Theorem \ref{prop2val}, we must have $\alpha_1 \alpha_2 = 2t$, contradicting that we are in case  $\alpha_1 \alpha_2=t-1$.

Now let us consider the other case $\alpha_1 \alpha_2=2t$.
If $k_{12} =n_2$ or $k_{21}=n_1$, then by Theorem~\ref{prop2val} and Lemma~\ref{Petersen} the graph, $\Gamma$ is the Petersen cone.
Hence we can assume that there are vertices $x \in V_1$ and $y \in V_2$ such that $x \nsim y$.
By Lemma~\ref{lem:bound1}, it follows that $(\alpha_1-\alpha_2)\alpha_2 \leqslant t$, and this implies that $\alpha_2^2 \geqslant t$, as $\alpha_1 \alpha_2 = 2t$.

%
%
Using the fact that $\alpha_1\alpha_2 = 2t$, it follows from Theorem~\ref{prop2val} that
\begin{equation}\label{10}
n= \frac{(\theta_0-1)(\alpha_1^2+\alpha_2^2+3t-\theta_0)}{2t}
\end{equation}
holds.

Since $\alpha_1\alpha_2=2t$ , and $\alpha_2^2 \geqslant t$, it follows that $4t+1 \leqslant \alpha_1^2+\alpha_2^2 \leqslant 5t$ and hence, by Eq. \eqref{10}, we have $n < 8t$.
 By Theorem \ref{prop2val} we have $\theta_0 \leqslant k_2 \alpha_1/\alpha_2 = \alpha_1 \alpha_2 +t \alpha_1/\alpha_2 \leqslant 2t + 2t = 4t$.
Therefore $m_2=(n-1+\theta_0)/(1+t)< (8t-1+4t)/(1+t)<12$ and by Theorem~\ref{thm:bell} we deduce $n \leqslant 12 \cdot 13/2 = 78$.

Now, using Theorem~\ref{prop2val} and Lemma~\ref{lem:th2bound} we can compute all the feasible valency-arrays and spectra for graphs satisfying these conditions (see the appendix).
We obtain four valency-arrays and spectra corresponding to the Petersen cone, the Van Dam-Fano graph, and the two valency-arrays and spectra from Lemma~\ref{lem:paramnonexi}.
\end{proof}

Now we can strengthen Theorem~\ref{prop2val} further.

\begin{theorem}\label{thm:k21}
	Let $\Gamma$ be a non-bipartite biregular graph in $\mathcal G(\theta_0,\theta_1,\theta_2)$ with valencies $k_1 > k_2$ whose valency partition has quotient matrix $\left ( \begin{smallmatrix}
	k_{11} & k_{12}\\
	k_{21} & k_{22}
	\end{smallmatrix} \right )$.
	The following are equivalent:
	\begin{enumerate}[(i)]
		\item $k_{21} = 1$;
		\item $k_{21} = n_1$;
		\item $k_{12} = n_2$;
		\item $\Gamma$ is a cone over a strongly regular graph.
	\end{enumerate}
\end{theorem} 
\begin{proof}
	Assume (i) holds.
	Let $x$ be a vertex of valency $k_2 = k_{21} + k_{22}$.
	Since $k_{21} = 1$, each vertex in $V_2$ is adjacent to precisely one vertex in $V_1$.
	Hence the sum $\sum_{y \in V_1} \nu_{x,y}$ is equal to $k_{11} + k_{22}$. 
	Therefore $k_2 + k_{11} = 1+\sum_{y \in V_1} \nu_{x,y} = n_1 \alpha_1 \alpha_2 + \theta_1 + \theta_2 + 1$.
	We also have that $k_2 = \alpha_2^2 - \theta_1\theta_2$.
	It follows that $n_1 = (\alpha_2^2 - (\theta_1 + 1)(\theta_2 + 1) + k_{11})/\alpha_1\alpha_2$.
	By Proposition~\ref{secondev1} we can assume that $\theta_1 \geqslant 2$ and, since Van Dam has classified such graphs with $\theta_2 \geqslant -2$, we can assume that $\theta_2 \leqslant -3$.
	Using these assumptions together with the inequality $k_{11} \leqslant n_1-1$, we can write 
	\begin{equation}
		\label{eqn:ineqn1}
		n_1 \leqslant \frac{\alpha_2^2 - (\theta_1+1)(\theta_2+1) - 1}{\alpha_1\alpha_2 - 1}.
	\end{equation}
	First observe that $\alpha_2^2 \leqslant \alpha_1\alpha_2 - 1$.
	By Theorem~\ref{prop2val} (iii), either $\alpha_1\alpha_2 = -\theta_2(\theta_1+1)$ or $\alpha_1\alpha_2 = -\theta_1(\theta_2+1)$.
	Applying \eqref{eqn:ineqn1} with $\alpha_1\alpha_2 = -\theta_2(\theta_1+1)$ gives $n_1 < 2$ and hence $n_1 = 1$, giving (ii).
	Otherwise, applying \eqref{eqn:ineqn1} with $\alpha_1\alpha_2 = -\theta_1(\theta_2+1)$ gives $n_1 < 3$.
	If $n_1 = 2$ then $\det Q = \theta_0 \theta_1 > 0$.
	But in this case $k_{11} \in \{0,1\}$ and for either value of $k_{11}$ we obtain $\det Q < 0$, a contradiction.
	Therefore $n_1 = 1$, as in (ii).
	The rest follows from the proof of Theorem~\ref{prop2val}.
\end{proof}

\section{Biregular graphs from designs} 
\label{sec:graphs_from_designs}


The results of this section link some highly structured graphs to certain designs.
A set $V$ of cardinality $v$ together with a collection of $k$-subsets $B_1,\dots,B_b$ (called blocks) of $V$ is called a $(v,b,r,k)$-\textbf{configuration} if each point of $V$ occurs in precisely $r$ blocks.
A $(v,b,r,k)$-configuration $D$ is called a \textbf{group divisible} design (GDD) if its $v$ points can be partitioned into $m$ sets $T_1, \dots, T_m$ each of size $n \geqslant 2$ such that any two points $x \in T_i$ and $y \in T_j$ occur together in $\lambda_1$ blocks if $i=j$ and $\lambda_2$ blocks otherwise.
We say that $D$ is a GDD with parameters $(v,b;r,k;\lambda_1,\lambda_2;m,n)$ and in Table~\ref{tab:fea} we write $\GDD(v,b;r,k;\lambda_1,\lambda_2;m,n)$.
(We refer to Bose~\cite{Bos77} for more details about GDDs.)
Note that if $n=v$ then $D$ is a $2$-design with parameters $(v,b,r,k,\lambda_1)$ or, for short, a $(v,k,\lambda_1)$-design.

\begin{theorem}\label{thm:2-des}
	Let $\Gamma$ be a graph in $\mathcal G(\theta_0,\theta_1,\theta_2)$ 
	such that its adjacency matrix $A$ under the valency partition has block form:
	\[
	A =
	\begin{pmatrix}
	A_1 & B_1\\
	B_2 & A_2
	\end{pmatrix}
	\text{~~with quotient matrix~~}
	\begin{pmatrix}
	k_{11} & k_{12}\\
	k_{21} & k_{22}
	\end{pmatrix}
	\] 
	and set $n_i = |V_i|$.
	\begin{enumerate}[(i)]
		\item Suppose that, for some $i \in \{1,2\}$, the induced subgraph on $V_i$ is a complete $m$-partite graph for some $m$.
	Set $\lambda_1 = \alpha_i^2-k_{ii}$ and $\lambda_2 = \alpha_i^2+\theta_1+\theta_2-k_{ii}+n_i/m$.
	
		\item Suppose that, for some $i \in \{1,2\}$, the induced subgraph on $V_i$ is the disjoint union of cliques $K_{n_i/m}$ for some $m$.
	Set $\lambda_1 = \alpha_i^2+\theta_1+\theta_2 + 1-k_{ii}$ and $\lambda_2 = \alpha_i^2$.
	\end{enumerate}
	In either of the above cases, set $j$ such that $\{i,j\} = \{1,2\}$, then the matrix $B_i$ 
	is an incidence matrix of a GDD with parameters $(n_i,n_j;k_{ij},k_{ji};\lambda_1,\lambda_2;m,n_i/m)$.
\end{theorem}
\begin{proof} 
	Suppose we are in the first case of the theorem, so that the induced subgraph, $K$, on $V_i$ is a complete $m$-partite graph with parts $T_1,\dots,T_m$.
	For any two vertices $x$ and $y$ in $V_i$, we have $\nu_{x,y} = (\theta_1+\theta_2)A_{xy} + \alpha_i^2$.
	Now, $x$ is adjacent to $y$ if and only if they are both in the same part $T_{l}$ for some $l$.
	Moreover, the number of common neighbours of $x$ and $y$ in $K$ is $k_{ii} - n_i/m$ if they are adjacent and $k_{ii}$ otherwise.
	Hence, taking $V_i$ as the set of points and $V_j$ as the set of blocks, we have that $B_i$ is the incidence matrix the required GDD.
	The other case follows similarly.
\end{proof}

\begin{remark}
	\label{rem:2-des}
	We can apply Theorem~\ref{thm:2-des} to some of the valency-arrays and spectra given in Table~\ref{tab:fea}.
	In the comment column of the table, we write the parameters of the GDD or $2$-design whose existence is implied by the existence of a graph having the corresponding valencies and spectra.
\end{remark}

\begin{corollary}\label{cor:2-des}
	Let $\Gamma$ be a graph in $\mathcal G(\theta_0,\theta_1,\theta_2)$ whose valency partition has quotient matrix $\left ( \begin{smallmatrix}
	k_{11} & k_{12}\\
	k_{21} & k_{22}
	\end{smallmatrix} \right )$.
	If both $k_{ii} = 0$ and $k_{jj} = n_j-1$ then $\Gamma$ exists if and only if there exists an $(n_i,k_{ji},\alpha_i^2)$-design.
\end{corollary}
\begin{proof}
	By Theorem~\ref{thm:2-des} it follows that the existence of the graph implies that of the design.
	Conversely, let $B$ be the incidence matrix of an $(n,k,\lambda)$-design with $b$ blocks such that each point is in $r$ blocks.
	Then \[
		A =
		\begin{pmatrix}
		O & B\\
		B^\top & J-I
		\end{pmatrix}
		\text{~~with quotient matrix~~}
		\begin{pmatrix}
		0 & r\\
		k & b-1
		\end{pmatrix}
	\]
	is an adjacency matrix of a graph and $A^2 + A + (\lambda-r)I = \alpha\alpha^T$ for some (positive) vector $\alpha$ with $\alpha_1^2 = \lambda$.
\end{proof}

Define the \textbf{total graph} of a $(q^3,q^2,q+1)$-design as the bipartite incidence graph with edges added between two blocks if they intersect in $q$ points (see \cite[p.\ 197]{MK}).
A resolvable $(n,k,\lambda)$-design is called \textbf{affine} if there exists some $\mu$ such that any two non-parallel blocks intersects in $\mu$ points. 
We write $\AR(n,k,\lambda)$ to refer to such a design in Table~\ref{tab:fea}.
We refer the reader to Shrikhande~\cite{shr} for a survey on such designs.

Van Dam~\cite[p.\ 104]{Dam3ev} showed that, given any affine resolvable $(q^{3},q^{2},q+1)$-design, the total graph is nonregular with spectrum $\{[q^{3}+q^{2}+q]^{1} ,[q]^{q^{3}-1} ,[-q]^{q^{3}+q^{2}+q}\}$.
In our next result we show that graphs having the valency-array and spectra of the total graph $\Gamma$ of such a design are necessarily isomorphic to $\Gamma$.

\begin{theorem}\label{thm:resaff}
	Let $q$ be a positive integer.
	Let $\Gamma$ be a biregular graph with valency-array $(q^{3}+q^{2}+q, q^{3}; q^{3} + 2 q^{2}, q^{2}+q+1 )$ and spectrum $\{[q^{3}+q^{2}+q]^{1} ,[q]^{q^{3}-1} ,[-q]^{q^{3}+q^{2}+q}\}$.
	Then $\Gamma$ is the total graph of some affine resolvable $(q^{3},q^{2},q+1)$-design.
\end{theorem}
\begin{proof}
	The adjacency matrix of $\Gamma$, under the valency partition, has block form:
		\[
		\begin{pmatrix}
		A_1 & B^{\top}\\
		B & A_2
		\end{pmatrix}
		\text{~~with quotient matrix~~}
		\begin{pmatrix}
		q^{3}+q^{2} & q^{2} \\
		q^{2}+q+1   & 0
		\end{pmatrix}.
		\]
     By Theorem~\ref{thm:2-des}, $B$ is the incidence matrix of a $(q^{3},q^{2},q+1)$-design.
	 Take $V_2$ as the set of points and $V_1$ as the set of blocks.
     The induced subgraph $\regSub$ on the vertex set $V_1$ is a $q^{3}+q^{2}$-regular graph with $n_1 =q^{3}+q^{2}+q$ vertices and, by interlacing, this subgraph has eigenvalue $-q$ with multiplicity at least $q^{2}+q$.
	 Hence the complement of $\regSub$ is a $q-1$-regular graph, which has eigenvalue $q-1$ with multiplicity at least $q^{2}+q+1$.
	 It follows that $\regSub$ is a complete $q^{2}+q+1$-partite graph $K_{q,\dots,q}$.
     So any blocks in the same part are disjoint and two blocks in different part intersect in $q$ points.
	 Therefore, our $(q^{3},q^{2},q+1)$-design is affine resolvable.
\end{proof}

\begin{remark}
	\label{rem:66}
	In Table~\ref{tab:fea}, the entry with $66$ vertices and largest eigenvalue $39$ satisfies the assumptions of Theorem~\ref{thm:resaff}.
	Indeed, it corresponds to an affine resolvable $(27,9,4)$-design of which there are known~\cite{Tonch} to exist precisely $68$.
\end{remark}

\section{Complements and switchings}
\label{sec:comp}

In this section we will discuss connected graphs $\Gamma$ that, together with their complements $\overline{\Gamma}$, both have precisely $3$ distinct eigenvalues.
Strongly regular graphs that are both connected and coconnected satisfy this property.
We will show that, for nonregular graphs, such a graph must be biregular.
In this section we will also consider biregular graphs with three distinct eigenvalues such that the property of having three distinct eigenvalues is preserved under switching with respect to the valency partition.
(Recall that \textbf{switching} a subset $W$ of vertices of a graph is the process of swapping each edge between $W$ and $V(\Gamma)\backslash W$ with a non-edge and \emph{vice versa}.)
First we describe the eigenvalues of the complement of a biregular graph with three eigenvalues.

\begin{proposition}\label{pro:bicomp}
	Let $\Gamma$ be a biregular $n$-vertex graph in $\mathcal G(\theta_0,\theta_1,\theta_2)$ and let $A$ be its adjacency matrix, and let $\theta_0+\theta$ be the trace of the quotient matrix of the valency partition of $\Gamma$.
	Then the complement $\overline \Gamma$ of $\Gamma$ has at most $4$ distinct eigenvalues.
	For $n > 6$ the eigenvalues of $\overline \Gamma$ are $-1-\theta_1$, $-1-\theta_2$, and
	\[
		\frac{n-2-(\theta_0+\theta)}{2} \pm \frac{\sqrt{ (n+\theta_0+\theta)^2-4( \theta_0\theta +\operatorname{tr}A^2)}}{2}.
	\]
	Moreover, for $i \in \{1,2\}$, the multiplicity of $-1-\theta_i$ in $\overline \Gamma$ is at least $m_i-1$ if $\theta_i = \theta$ or at least $m_i$ if $\theta_i \ne \theta$.
\end{proposition}
\begin{proof}
	By Theorem~\ref{prop2val}, the valency partition of $\Gamma$ is equitable with quotient matrix $Q$.
	Therefore the valency partition of $\overline \Gamma$ is also equitable with quotient matrix $\overline Q$.
	For $n>6$, by Theorem~\ref{thm:bell} both the multiplicities $m_1$ and $m_2$ are at least $2$.
	Now $\overline \Gamma$ has eigenvalues $-1- \theta_1$ and $-1-\theta_2$ with multiplicities at least $m_1 - 1$ and $m_2 -1$ respectively.
	The two eigenvalues of $\overline Q$ are determined from the trace and determinant of $Q$.
	These eigenvalues have the form
	\[
		\frac{n-2-(\theta_0+\theta)}{2} \pm \frac{\sqrt{ (n+\operatorname{tr} Q)^2-4( \det Q +\operatorname{tr}A^2)}}{2}.
	\]
	Since the sum of the eigenvalues of $\overline \Gamma$ is zero, we obtain that the remaining eigenvalue is $\operatorname{tr} Q-(\theta_0+\theta_1+\theta_2) - 1$.
	By Theorem~\ref{prop2val}, we have either $\operatorname{tr} Q = \theta_0 + \theta_1$ or $\operatorname{tr} Q = \theta_0 + \theta_2$, whence the remaining eigenvalue is equal to $-1-\theta_2$ or $-1-\theta_1$ respectively. 
\end{proof}

\begin{theorem}\label{thm:complement} 
	Let $\Gamma$ be an $n$-vertex graph in $\mathcal G(\theta_0,\theta_1,\theta_2)$ such that its complement $\overline{\Gamma}$ is in $\mathcal G(\theta_0^\prime,\theta_1^\prime,\theta_2^\prime)$.
	Then $\Gamma$ is biregular with valencies $k_1$ and $k_2$ satisfying
	\[
		k_1, k_2 =\frac{n+\theta_1+\theta_2\pm\sqrt{(n+\theta_1+\theta_2+2\theta_1\theta_2)^{2}-4\theta_2^{2}(\theta_1+1)^{2}}}{2}.
	\]
	Moreover,
		\begin{align*}
			n &= \frac{(\theta_0-\theta_1)^{2}}{\theta_0-\theta_1+\theta_1\theta_2+\theta_2}; &&\theta_0 = \frac{n}{2}+\theta_1\pm\sqrt{n(n+4\theta_2(\theta_1 +1))}/2; \\
			\theta_0' &= n-1 -\theta_0 + \theta_1-\theta_2, &&\theta_1^\prime 
			= -1 -\theta_2 \text{, and }\; \theta_2^\prime = -1 -\theta_1.
		\end{align*}
\end{theorem}
	
\begin{proof}
	We first show that $\Gamma$ is biregular.
	To distinguish between $\Gamma$ and $\overline \Gamma$ we denote elements of $\overline \Gamma$ with an apostrophe, e.g., we denote the valency of a vertex $x$ in $\overline \Gamma$ by $d_x^\prime$ and its eigenvalues $\theta_i^\prime$ have multiplicity $m_i^\prime$.

	Firstly, if either $m_1$ or $m_2$ are in $\{1,2\}$, then, by Theorem~\ref{thm:bell}, either $\theta_1=0$ or $n\leqslant {3\cdot4}/2=6$.
	The only graphs satisfying this condition are the complete bipartite graphs (see Corollary~\ref{cor:compbi}), and these graphs do not satisfy the assumption of the theorem.
	Therefore we can assume that both $m_1 \geqslant 3$ and $m_2 \geqslant 3$, and hence, for each $i \in \{1,2\}$, the complement $\overline{\Gamma}$ has eigenvalue $-1-\theta_i$ with multiplicity at least $m_i-1$.
	Thus the two smaller eigenvalues of $\overline{\Gamma}$ are $\theta_1'=-1-\theta_2$ with multiplicity $m_1'\geqslant m_2-1$ and $\theta_2'=-1-\theta_1$ with multiplicity $m_2'\geqslant m_1-1$.

Since $d_x+d_x'=n-1$, by Lemma~\ref{lem:rhobound}, for all $x \in V(\Gamma)$ we have
\begin{equation}
\theta_0 \geqslant \frac{\sum d_x}{n} \quad \text{ and } \quad \theta_0' \geqslant \frac{\sum d_x'}{n}.\label{eqn:bound1}
\end{equation}
It follows that $\theta_0+\theta_0'\geqslant \frac{n\cdot(n-1)}{n}=n-1$.
Further, we have $m_1+m_2=m_1'+m_2'=n-1$ and
\begin{align}
\theta_0+m_1\theta_1+m_2\theta_2 = 0 \quad \text{ and } \quad
\theta_0'+m_1'\theta_1'+m_2'\theta_2' = 0.  \label{eqn:tr1}
\end{align}
Hence $\theta_1(m_1-m_2')+\theta_2(m_2-m_1')=n-1-(\theta_0+\theta_0')$.
Thus $\theta_1(m_1-m_2')+\theta_2(m_2-m_1') \leqslant 0$ which implies either $m_1=m_2'$ and $m_2=m_1'$, or $m_1=m_2'-1$ and $m_2=m_1'+1$.

We will first consider the case $m_1=m_2'$ and $m_2=m_1'$.
In this case we have equality in \eqref{eqn:bound1} which implies that $\Gamma$ is regular (in fact strongly regular), but $\Gamma$ is nonregular so we must have that $m_1=m_2'-1$ and $m_2=m_1'+1$.

Take two vertices $x,y$.
Since $\Gamma$ and $\overline{\Gamma}$ are both connected, we can assume $x \sim y$ in $\Gamma$.
Denote by $N_x$ the set of neighbours of $x$ that are not adjacent to $y$ in $\Gamma$.
Then we can write
\begin{align*}
	|N_x| = d_x-\nu_{x,y}-1 &= \alpha_x^{2}-\theta_1\theta_2-\alpha_x\alpha_y-\theta_1-\theta_2-1 \\
	&= \alpha_x(\alpha_x-\alpha_y)-(\theta_1+1)(\theta_2+1).
\end{align*}

On the other hand, all vertices in $N_x$ are the neighbours of $y$ that are not adjacent to $x$ in $\overline{\Gamma}$, and hence
\begin{align*}
	|N_x| = d_y'-\nu_{x,y}^\prime &= \alpha_y'^{2}-(\theta_1+1)(\theta_2+1)-\alpha_x'\alpha_y' \\
	&= \alpha_y'(\alpha_y'-\alpha_x')-(\theta_1+1)(\theta_2+1).
\end{align*}

Then for all pairs of vertices $x$ and $y$, we have the following equations
\begin{align}
\alpha_x(\alpha_x-\alpha_y) &= \alpha_y'(\alpha_y'-\alpha_x'),  \label{eqn:aleq1} \\
\alpha_y(\alpha_y-\alpha_x) &= \alpha_x'(\alpha_x'-\alpha_y').  \label{eqn:aleq2}
\end{align}
Now, taking the sum and difference of Eqs. \eqref{eqn:aleq1} and \eqref{eqn:aleq2} implies that ${(\alpha_x-\alpha_y)}^{2}={(\alpha_y'-\alpha_x')}^{2}$ and $\alpha_x^{2}-\alpha_y^{2}=\alpha_y'^{2}-\alpha_x'^{2}$.
Whence we have $\alpha_x=\alpha_y'$ and $\alpha_y=\alpha_x'$ for all pairs of vertices $x$ and $y$ with $d_x \ne d_y$.
Therefore $\Gamma$ is biregular.

Since $m_1=m_2^\prime -1$ and $m_2 = m_1^\prime + 1$, we see from Eq.~\eqref{eqn:tr1} that $\theta_0+\theta_0'=n-1+\theta_1-\theta_2$.

	By Proposition~\ref{pro:bicomp}, the quotient matrix of the valency partition of $Q$ has eigenvalues $\theta_0$ and $\theta_2$.
	Applying Theorem~\ref{prop2val} gives $\alpha_1\alpha_2=-\theta_2(\theta_1+1)$, and hence $\alpha_1^{2}+\alpha_2^{2}=n+2\theta_1\theta_2+\theta_1+\theta_2$.
	Therefore we have
	$$
	k_{1}, k_{2}
	= \frac{n+\theta_1+\theta_2}{2} \pm \sqrt{\Big (\frac{n+\theta_1+\theta_2}{2}+\theta_1\theta_2 \Big )^{2}-\theta_2^{2}(\theta_1+1)^{2}}.
	$$
	Moreover, again by Theorem~\ref{prop2val}, we can write $n=\frac{(\theta_0-\theta_1)^{2}}{\theta_0-\theta_1+\theta_1\theta_2+\theta_2}$, and hence $\theta_0=\frac{n}{2}+\theta_1\pm\sqrt{n(n+4\theta_2(\theta_1 +1))}/2$.
\end{proof}
%
	
	In the next result we describe an infinite family of feasible valency-arrays and spectra for biregular graphs with three distinct eigenvalues whose complements also have three distinct eigenvalues.

	\begin{proposition}\label{pro:comp}
			For an integer $t \geqslant 1$, set $\theta_1 = -\theta_2 = 2t^{2}+2t-1$, $\theta_0=\theta_1+2\theta_1(\theta_1 +1)$.
			Further set $n_1 = n_2 = 2\theta_1(\theta_1+1)$ and $k_1, k_2 = 2(\theta_1+1)\theta_1 \pm (2t+1)\theta_1$.
			Suppose $\Gamma$ has valency-array $(n_1,n_2; k_1, k_2)$ and spectrum $\{ [\theta_0]^1, [\theta_1]^{m_1}, [\theta_2]^{m_2} \}$.
			Then the complement $\overline \Gamma$ is in $\mathcal G(\theta_0 -1,-1 - \theta_2, -1 -\theta_1)$.
			Moreover, the valency-array and spectra are feasible.	
	\end{proposition}
	\begin{proof}
		From the assumptions of the proposition we can write $\alpha_1\alpha_2 = \theta_1(\theta_1+1)$ and $\{ \alpha_1^2, \alpha_2^2 \} = \{ 2(2t^4+6t^3+5t^2-1), 2t^2(2t^2+2t-1) \}$.
		By Theorem~\ref{prop2val} the trace of the quotient matrix of the valency partition is $\theta_0 + \theta_2$.
		Apply Proposition~\ref{pro:bicomp}.
		It then suffices to show that
			\[
				\theta_0 - 1, -1 -\theta_1 = \frac{n-2-(\theta_0+\theta_2)}{2} \pm \frac{\sqrt{ (n+\theta_0+\theta_2)^2-4( \theta_0\theta_2 +\operatorname{tr}A^2)}}{2}.
			\]
		This equation follows from straightforward algebraic manipulations, using the multiplicity equations \eqref{mult} and the formula for $n$ in Theorem~\ref{prop2val}.
		
		To show that the valency-array and spectrum are feasible, one needs to check that Theorem~\ref{prop2val} is satisfied such that $Q$ is a nonnegative integer matrix and the quantities $n_1$, $n_2$, $m_1$, and $m_2$ are positive integers.
		We leave this task to the reader.
	\end{proof}

\begin{remark}
\label{rem:48}
	In Proposition~\ref{pro:comp}, the case having smallest $n$ is obtained when $t=1$.
	This gives $n=48$ with spectrum $\{[27]^{1},[3]^{19},[-3]^{28} \}$, and its complement has spectrum $\{ [26]^{1},[2]^{27},[-4]^{20} \}$.
	In Theorem~\ref{thm:48}, below, we show that there exists no graph satisfying these conditions.
\end{remark}
	
	So far we know of no graphs that satisfy Theorem~\ref{thm:complement}; it is an open problem to determine their existence.

	    Let $\Gamma$ be a graph in $\mathcal G(\theta_0,\theta_1,\theta_2)$ with two valencies (say $k_1,k_2$).
		We can also consider the graph $\Gamma^\prime$ obtained by switching with respect to $V_1$ in $\Gamma$.
		Let $Q$ and $Q^\prime$ be the quotient matrices of the valency partitions of $\Gamma$ and $\Gamma^\prime$ respectively,
	    $Q = \begin{pmatrix} k_{11} & k_{12} \\ k_{21} & k_{22} \end{pmatrix}$ and $Q' = \begin{pmatrix} k_{11} & n_2 - k_{12} \\ n_1 - k_{21} & k_{22} \end{pmatrix}$.
		In the proposition below, we show that $\Gamma^\prime$ has at most $4$ distinct eigenvalues.
		
		\begin{proposition}\label{pro:biswitch}
			Let $\Gamma$ be a graph in $\mathcal G(\theta_0,\theta_1,\theta_2)$, 
			so that its adjacency matrix $A$ under the valency partition has block form:
			\[
			A = \begin{pmatrix}
			A_1 & B^{\top}\\
			B & A_2
			\end{pmatrix}
			\text{~~with quotient matrix~~}
			Q =\begin{pmatrix}
			k_{11} & k_{12}\\
			k_{21} & k_{22}
			\end{pmatrix},
			\]
			and let $\Gamma^\prime$ be obtained by switching the vertices with respect to $V_1$.
			Then $\Gamma^\prime$ has at most $4$ distinct eigenvalues.
			Namely, $\theta_1$, $\theta_2$, and
			\[
				\frac{k_{11}+k_{22}}{2} \pm \frac{\sqrt{ (k_{11}-k_{22})^2+4(n_1-k_{21})(n_2-k_{12})}}{2}.
			\]
			Moreover, for $k_{11} + k_{22} = \theta_0 + \theta$ and $i \in \{1,2\}$, the multiplicity of $\theta_i$ in $\overline \Gamma$ is at least $m_i-1$ if $\theta_i = \theta$ or at least $m_i$ if $\theta_i \ne \theta$.
		\end{proposition}
		\begin{proof}
			By Theorem~\ref{prop2val} the valency partition is equitable.
			Apply Corollary 3.2 of Muzychuk and Klin~\cite{MK} to find that $\Gamma^\prime$ has eigenvalues $\theta_1$ and $\theta_2$ where the multiplicity of $\theta_i$ is at least $m_i-1$ if $\theta_i = \theta$ or at least $m_i$ if $\theta_i \ne \theta$.
			Moreover, the remaining two eigenvalues are the eigenvalues of the quotient matrix $Q^\prime$ of the corresponding equitable partition of $\Gamma^\prime$,
				$\begin{pmatrix} 
					k_{11} & n_2 - k_{12} \\ 
					n_1 - k_{21} & k_{22} 
				\end{pmatrix}$.
		\end{proof}
    	
		Suppose that $\Gamma^\prime$ has precisely three distinct eigenvalues.
		By Proposition~\ref{pro:biswitch}, $\Gamma^\prime$ is in $\mathcal G(\theta_0^\prime, \theta_1,\theta_2)$.
	    On the one hand, if $\theta_0' = \theta_0$, then $Q$ and $Q'$ have same eigenvalues.
		Hence we have the equality $\det Q = \det Q' = k_{11} k_{22} -k_{12} k_{21} = k_{11} k_{22} -(n_1 - k_{21})(n_2 - k_{12})$ and thus, $n_2 = 2k_{12}$ and $n_1 = 2k_{21}$.
    
	    On the other hand, if $\theta_0' \neq \theta_0$ then, without loss generality, we can assume that $Q$ has eigenvalues $\theta_0$ and $\theta_2$, and that $Q'$ has eigenvalues $\theta_0'$ and $\theta_1$.
	
		We consider the special case $k_{11} = k_{22}$, $\alpha_1 = s\alpha_2$, and $\theta_2 = -(s-1)\theta_1-s$.
		Similar to Proposition~\ref{pro:comp}, above, we describe an infinite family of feasible valency-arrays and spectra for biregular graphs with three eigenvalues such that switching with respect to $V_1$ gives another graph having three distinct eigenvalues.
		\begin{proposition}\label{pro:switch}
			For integers $s \geqslant 2$ and $t \geqslant 1$, set $\theta_1 = st$, $\theta_2 = -(s-1)\theta_1-s$, $\theta_0 = s(2st+1)(st+1-t)$.
			Further set $n_1 = (2st+1)(st-t+1)-t$, $n_2 = s^2 n_1$, $k_1=(s^{2}(st+1)+s^{2}t)(st-t+1)$, and $k_2 = (s^{2}t+st+1)(st-t+1)$.
			Suppose $\Gamma$ has valency-array $(n_1,n_2; k_1, k_2)$ and spectrum $\{ [\theta_0]^1, [\theta_1]^{m_1}, [\theta_2]^{m_2} \}$ and let $\Gamma^\prime$ be the graph obtained by switching the vertices in $V_1$ of $\Gamma$.
			Then $\Gamma^\prime$ is in $\mathcal G(\theta_0 -s(st+1),\theta_1,\theta_2)$.
			Moreover, the valency-array and spectrum are feasible.
		\end{proposition}
		\begin{proof}
			By Theorem~\ref{prop2val} the trace of the quotient matrix of the valency partition is $\theta_0 + \theta_2$.
			Then by Proposition~\ref{pro:biswitch} it suffices to show that the largest eigenvalue of $\Gamma^\prime$ is $\theta_0 -s(st+1)$, that is, we want to show
			\[
				\theta_0 -s(st+1), \theta_1 = \frac{k_{11}+k_{22}}{2} \pm \frac{\sqrt{ (k_{11}-k_{22})^2+4(n_1-k_{21})(n_2-k_{12})}}{2}.
			\]
			Using Theorem~\ref{prop2val}, we can write each of the terms $k_{11}$, $k_{12}$, $k_{21}$, $k_{22}$, $n_1$, and $n_2$ in terms of $s$ and $t$. 
			The equality then follows from straightforward algebraic manipulations.
			
			To show that the valency-array and spectrum are feasible, one needs to check that Theorem~\ref{prop2val} is satisfied such that $Q$ is a nonnegative integer matrix and the quantities $n_1$, $n_2$, $m_1$, and $m_2$ are positive integers.
			We leave this task to the reader.
		\end{proof}
	   		%
	\begin{remark}
		In Proposition~\ref{pro:switch}, the case having smallest $n$ is obtained when $s=2$ and $t=1$.
		This gives $n=45$ with spectrum $\{[20]^{1},[2]^{26},[-4]^{18} \}$, and $\Gamma'$ has spectrum $\{ [14]^{1},[2]^{27},[-4]^{17} \}$.
		 (See Table~\ref{tab:fea} for this case and the case on 80 vertices.)
    \end{remark}
	
\section{Constructions using star complements} 
\label{sec:computational_constructions}

In this section we describe constructions for certain graphs with three eigenvalues using so-called `star complements'.
We also examine the structure of biregular graphs in $\mathcal G(\theta_0,\theta_1,\theta_2)$ where $\theta_1 + \theta_2 = -1$.
We find that such graphs have properties which make it convenient to use star complements to attempt to construct them. 
In particular we show the existence of a case which was an open problem from Van Dam's paper~\cite{Dam3ev}.
We also construct a new graph having three valencies and three eigenvalues.

\subsection{Star complements} 
\label{sub:star_complements}

	To describe our new constructions we first recall the notion of the \emph{star complement}.
	Let $\theta$ be an eigenvalue of an $n$-vertex graph $\Gamma$ and suppose that 
	$\theta$ has multiplicity $m$.
	Define a \textbf{star set} for $\theta$ to be a subset $X \subset V(\Gamma)$
	such that $|X| = m$ and $\theta$ is not an eigenvalue of the graph induced on $V(\Gamma) - X$.
	Now we can state the \emph{Reconstruction Theorem} (see \cite[Theorems 7.4.1 and 7.4.4]{crs97}).

\begin{theorem}\label{thm:reconst}
	Let $X$ be a subset of vertices of a graph $\Gamma$ and suppose that $\Gamma$ has adjacency matrix
	\[
		\begin{pmatrix}
			A_X & B^\top \\
			B & C
		\end{pmatrix},
	\]
	where $A_X$ is the adjacency matrix of the subgraph induced by $X$.
	Then $X$ is a star set for $\theta$ if and only if $\theta$ is not an eigenvalue of $C$ and
	$\theta I - A_X = B^\top (\theta I - C)^{-1} B$.
\end{theorem}

The graph $\regSub$ induced on $V(\Gamma) - X$ (having adjacency matrix $C$ in Theorem~\ref{thm:reconst}) is called the \textbf{star complement} of $\theta$.
Star sets and star complements exist for any eigenvalue and any graph and moreover, for $\theta \not \in \{0,1\}$, it can be shown that $\regSub$-neighbourhoods of the vertices of $X$ are non-empty and distinct~\cite[Chapter 7]{crs97}.
For vectors in $\mathbf{v}, \mathbf{w} \in \Z^{n-m}$, define the bilinear map $\inprod{\mathbf{v},\mathbf{w}} := \mathbf{v}^\top (\theta I - C)^{-1} \mathbf{w}$.
Let $V$ be the set of vectors $\mathbf{v} \in \{0,1\}^{n-m}$ satisfying $\inprod{\mathbf{v},\mathbf{v}} = \theta$.
Form a graph having $V$ as its vertex set where two vectors $\mathbf{v}$ and $\mathbf{w}$ are adjacent if $\inprod{\mathbf{v},\mathbf{w}} \in \{0, -1\}$.
This graph is known as the \textbf{compatibility graph} for $\regSub$.
Cliques of size $m$ in the compatibility graph then give the columns of the matrix $B$ as in Theorem~\ref{thm:reconst}.

\begin{theorem}\label{thm:30}
	There exist at least $21$ graphs having valency-array
	$(15,15;14,8)$ and spectrum $\{ [12]^1, [2]^{15}, [-3]^{14} \}$.
\end{theorem}
\begin{proof}
	Let $\Gamma$ be a graph having the assumed valency-array and spectrum so that its valency partition has the quotient matrix
	$\begin{pmatrix}
	10 & 4\\
	4 & 4
	\end{pmatrix}$.
	Now we assume that the star complement of $\Gamma$ for the eigenvalue $2$ is the subgraph induced on either $V_1$ or $V_2$.
	Next we check all possibilities for these subgraphs using \texttt{Magma}~\cite{Magma} and \texttt{nauty}~\cite{nau}.
	Such graphs are regular with valency equal to either $4$ or $10$.
	By interlacing, any such star complement must have smallest eigenvalue at least $-3$ and second largest eigenvalue less than $2$.
	There are $94$ (resp. $43$) regular graphs on $15$ vertices with valency $10$ (resp. $4$), smallest eigenvalue at least $-3$, and second largest eigenvalue less than $2$.
	For each of these potential star complements, we construct the compatibility graph and search for cliques of size $15$.
	This process produced a list of $21$ non-isomorphic graphs with the assumed valency-array and spectrum.
\end{proof}

We remark that the existence of graphs with the valency-array and spectrum given in Theorem \ref{thm:30}, was an open case from Van Dam's paper \cite{Dam3ev}.

\begin{theorem}\label{thm:36}
	There are precisely two graphs having valency-array $(18,9,9; 24,14,8)$ and spectrum $\{ [20]^1, [2]^{17}, [-3]^{18} \}$.
\end{theorem}
\begin{proof}
	Let $\Gamma$ be a graph having the assumed valency-array and spectrum so that its valency partition has quotient matrix
	\[
		\begin{pmatrix}
		15 & 6 & 3 \\
		12 & 2 & 0 \\
		6 & 0 & 2
		\end{pmatrix}.
	\]
	Hence, the graphs induced on each of the subsets $V_2$ and $V_3$ are disjoint unions of cycles.
	On $9$ vertices, there are only $4$ graphs that are unions of cycles.
	Therefore, there are $16$ possible graphs for the subgraph $\regSub$ induced on the set $V_2 \cup V_3$.
	Each of these $16$ graphs is a potential star complement for the eigenvalue $-3$.
	For each of these potential star complements, we construct the compatibility graph and search for cliques of size $18$.
	This process produced a list of two non-isomorphic graphs with the assumed valency-array and spectrum.
	
	Now, unlike in the proof of Theorem~\ref{thm:30} (since the disjoint union of cycles does not have $-3$ as an eigenvalue), the graph induced on the set $V_2 \cup V_3$ must be a star complement for the eigenvalue $-3$.
	This means that there can exist no other graphs with the assumed valency-array and spectrum.
\end{proof}

One of the graphs, $\Gamma_1$, (say) satisfying the assumption of Theorem~\ref{thm:36} was constructed in \cite{cds99}.
The subgraphs of $\Gamma_1$ induced on each of $V_2$ and $V_3$ consist of three triangles (or $3$-cycles) and the subgraph of $\Gamma_1$ induced on $V_1$ consists of the complement of six triangles.
The other graph, $\Gamma_2$, (say) satisfying the assumption of Theorem~\ref{thm:36} was previously unknown.
The subgraphs of $\Gamma_2$ induced on $V_2$ and $V_3$ each consist of a triangle and a hexagon and the subgraph of $\Gamma_1$ induced on $V_1$ consists of the complement of two triangles and two hexagons.


\subsection{Graphs in $\mathcal G(\theta_0,\theta_1,\theta_2)$ with $\theta_1 + \theta_2 = -1$} 
\label{sub:_theta_1_theta_2_1}

In this section we examine biregular graphs in $\mathcal G(\theta_0,\theta_1,\theta_2)$ where $\theta_1 + \theta_2 = -1$.
We find that such graphs have properties which make it convenient to use star complements to attempt to construct them.
Compare Eq.~\eqref{eq-like-starcomp} to the equation in Theorem~\ref{thm:reconst}.
Indeed, above we used the star complement method to construct such graphs (see Theorem~\ref{thm:30}).
We denote by $\mathbf{j}$ the `all ones' (column) vector and we define the matrix $J:= \mathbf{j} \mathbf{j}^\top$.
First we will need a lemma from linear algebra. 

\begin{lemma}[See \cite{Buss00,Hae}]\label{two-ev-lemma}
Let $M$ be a symmetric $n \times n$ matrix with a symmetric partition 
\[
M=
\begin{pmatrix}
M_1 & N\\
N^{\top} & M_2\\
\end{pmatrix},
\]
where $M_1$ has order, say, $n_1$. 
Suppose that $M$ has just two distinct eigenvalues 
$r>s$, with multiplicities $f$ and $n-f$. 
Let $\eta_1\geqslant \dots \geqslant \eta_{n_1}$ be 
the eigenvalues of $M_1$ and let $\mu_1\geqslant \dots \geqslant \mu_{n-n_1}$ be the eigenvalues of $M_2$. 
Then $r\geqslant \eta_i\geqslant s$ for $i=1,\dots,n_1$, and 
\[
\mu_i =\left \{ \begin{aligned}
r{\rm~if~}1\leqslant i\leqslant f-n_1,\\
s{\rm~if~}f+1\leqslant i\leqslant n-n_1,\\
r+s-\eta_{f-i+1}{\rm~otherwise}.
\end{aligned}
\right.
\]
\end{lemma}

Now we give a structural result for graphs with a certain spectrum.

\begin{lemma}
Let $\Gamma$ be a connected graph with spectrum 
$\left \{ [\theta_0]^1, [\theta]^{n}, [-\theta-1]^{n-1} \right \}$, so that 
its adjacency matrix $A$ under the valency partition has block form:
\[
A=
\begin{pmatrix}
A_1 & B^{\top}\\
B & A_2
\end{pmatrix}
\text{~~with quotient matrix~~}
\begin{pmatrix}
k_{11} & k_{12}\\
k_{21} & k_{22}
\end{pmatrix},
\]
where $A_1$ and $A_2$ are both $n\times n$ matrices, i.e., $n=n_1=n_2$, where $n_i$ is the number 
of vertices with valency $k_i$. 


Then the matrices $A_1$ and $J-I-A_2$ are cospectral.
 
\end{lemma}
\begin{proof}
	
	We first recall that if $\regSub$ is a regular graph with eigenvalues 
	$\eta_0=k$ (its valency), $\eta_1, \eta_2, \dots$, then the complement of 
	$\regSub$ has eigenvalues $|\regSub|-1-k$ and  $-1-\eta_i$ for $i\geqslant 1$.
 
	For given number $c$, the matrix $M':=A+c\alpha\alpha^{\top}$ has eigenvalues 
	$\theta_0+c|\alpha|^2$ with multiplicity 1, and $\theta$ and $-\theta-1$ with multiplicities $n$ and $n-1$ respectively.
	Choose $c$ equal to $-(\theta+1+\theta_0)/ |\alpha|^2$. 
	Then the spectrum of $M'$ is $\left \{ [\theta]^n, [-\theta-1]^{n} \right \}$.
	Moreover, we can write $M'$ as follows:
	\[
	M' =
	\begin{pmatrix}
	A_1+c\alpha_1^2 J & B^{\top}+c\alpha_1\alpha_2 J\\
	B+c\alpha_1\alpha_2 J & A_2+c\alpha_2^2 J
	\end{pmatrix}
	.
	\]

	In the notation of Lemma~\ref{two-ev-lemma} applied to $M'$, we have 
	$v=2n$, $v_1=n$, $\{r,s\}=\{\theta,-\theta-1\}$, and $f=v-f=n$.
	Let $\eta_1\geqslant \dots \geqslant\eta_{n}$ be 
	the eigenvalues of $M_1:=A_1+c\alpha_1^2 J$ and let $\mu_1\geqslant \dots \geqslant\mu_{n}$ 
	be the eigenvalues of $M_2:=A_2+c\alpha_2^2 J$. 

	Since $A_1$ and $A_2$ are the adjacency matrices of regular graphs, 
	the matrices $M_1$ and $M_2$ have the same eigenvalues as the matrices 
	$A_1$ and $A_2$, respectively, except for the eigenvalues with eigenvector $\mathbf{j}$.

	By the conclusion of Lemma~\ref{two-ev-lemma}, 
	we see that 
	\[
	\mu_i=r+s-\eta_{f-i+1}=-1-\eta_{f-i+1},\text{~~for~}i=1,\dots,n.
	\]

	It now follows that $A_1$ and $J-I-A_2$ have the same eigenvalues, 
	except for the eigenvalues with eigenvector $\mathbf{j}$. 
	But these eigenvalues are determined from $\operatorname{tr} A_1=\operatorname{tr} A_2=0$, 
	and therefore they also coincide. The lemma is proved.	
\end{proof}

We can also prove the converse.

\begin{lemma} 
	\label{lem:n1n2}
Let $\Gamma$ be a graph in $\mathcal G(\theta_0,\theta_1,\theta_2)$, 
so that its adjacency matrix $A$ under the valency partition has block form:
\[
A =
\begin{pmatrix}
A_1 & B^{\top}\\
B & A_2
\end{pmatrix}
\text{~~with quotient matrix~~}
\begin{pmatrix}
k_{11} & k_{12}\\
k_{21} & k_{22}
\end{pmatrix},
\]
where $A_1$ and $J-I-A_2$ are cospectral.
Assume $n=n_1=n_2$, where $n_i$ is the number of vertices with valency $k_i$.

Then $k_{12}=k_{21}$ and $\Gamma$ has spectrum 
$\left \{ [\theta_0]^1, [\theta]^{n}, [-\theta-1]^{n-1} \right \}$, 
where $\theta\in \{\theta_1,\theta_2\}$ is an eigenvalue of the quotient matrix 
of the valency partition of $\Gamma$.
Moreover, $BA_1=(J-I-A_2)B$, and
\begin{equation}\label{eq-like-starcomp}
B^{\top}B=(\theta I-A_1)(\theta I-(J-I-A_1)),\text{~and~}
BB^{\top}=(\theta I-A_2)(\theta I-(J-I-A_2)).
\end{equation}
\end{lemma}

\begin{proof}
	First, in order to simplify the exposition below, we write $A_2=J-I-\tilde{A_1}$.
	Observe that $\tilde{A_1}$ is cospectral to $A_1$.
	Also note that $k_{12}=k_{21}$, since $n_1=n_2$ (see the proof of Theorem~\ref{prop2val}). 
	In particular, $B\mathbf{j}=B^{\top}\mathbf{j}=k_{12}\mathbf{j}$.
	Also note that $k_{11} + k_{22} = n-1 = \theta_0 + \theta$ and
	we can write
	\begin{equation}
		\label{eqn:M2}
		(A-\theta_1 I)(A-\theta_2 I)=A^2-(\theta_1+\theta_2)A+\theta_1\theta_2 I = \alpha \alpha^\top.
	\end{equation}

	For a matrix $X$ whose rows and columns are indexed by the vertex set of $\Gamma$, we will denote 
	by $X_{i,j}$ a submatrix of $X$ whose rows (columns) correspond to the vertices of $\Gamma$ of valency $k_i$ 
	($k_j$, respectively).
	We have
	\[
	A^2 =
	\begin{pmatrix}
	A_1^2+B^{\top}B & A_1B^{\top}-B^{\top}\tilde{A_1}+B^{\top}(J-I)\\
	BA_1-\tilde{A_1}B+(J-I)B & A_2^2+BB^{\top}
	\end{pmatrix}.
	\]

	On the other hand, by Eq. \eqref{eqn:M2}, we have
	\[
	(A^2)_{1,1} = A_1(\theta_1+\theta_2) + \alpha_1^2 J - \theta_1\theta_2 I,
	\]
	\[
	(A^2)_{2,2} = (J-I-\tilde{A_1})(\theta_1+\theta_2) + \alpha_2^2 J - \theta_1\theta_2 I,
	\]
	so that 
	\begin{equation}\label{eq-btb}
	B^{\top}B = -A_1^2 + A_1(\theta_1+\theta_2) + \alpha_1^2 J - \theta_1\theta_2 I,
	\end{equation}
	\begin{equation}\label{eq-bbt}
	BB^{\top} = -\tilde{A_1}(\tilde{A_1} + (2+\theta_1+\theta_2)I) + (\theta_1+\theta_2-n+2(1+k_{11})+\alpha_2^2) J - (\theta_1+1)(\theta_2+1) I,
	\end{equation}

	Note that $BB^{\top}\mathbf{j}=B^{\top}B\mathbf{j}=k_{12}k_{21}\mathbf{j}$.
	Multiplying Eqs. (\ref{eq-btb}) and (\ref{eq-bbt}) by $\mathbf{j}$, and 
	then comparing their right hand sides, we obtain that 
	\[
	(2k_{11}+1)(\theta_1+\theta_2+1) = n(\theta_1+\theta_2+1).
	\]

	Now if $(2k_{11}+1)=n$ holds then $k_{22}=n-1-k_{11}=k_{11}$ and, thus, $\Gamma$ is a regular graph, 
	a contradiction. Therefore $\theta_1+\theta_2=-1$.

	Suppose that $\Gamma$ had spectrum $\left \{ [\theta_0]^1, [\theta]^{m_1}, [-\theta-1]^{m_2} \right \}$, 
	where $1+m_1+m_2=|V(\Gamma)|=2n$. Then 
	\begin{equation*}
		\label{eq-tr-M}
	\theta_0+(m_1-m_2)\theta-m_2={\rm tr}(A)=0.
	\end{equation*}

	This, together with the equation $\theta_0+\theta=n-1$, gives 
	$n-m_1=2(m_1-n)\theta$. 
	Since $\theta$ is an integer, we see 
	that $n=m_1$ and hence $\Gamma$ has spectrum $\left \{ [\theta_0]^1, [\theta]^{n}, [-\theta-1]^{n-1} \right \}$.

	Substituting $\theta_1+\theta_2=-1$ into Eqs. (\ref{eq-btb}) and (\ref{eq-bbt}), 
	we see that
	\[
	B^{\top}B = -A_1^2 - A_1 + \alpha_1^2 J - \theta_1\theta_2 I,
	\]
	\[
	BB^{\top} = -\tilde{A_1}^2 - \tilde{A_1} + \alpha_1^2 J - \theta_1\theta_2 I.
	\]

	Further, let us show that $BA_1=\tilde{A_1}B$ holds. 
	From Eq. \eqref{eqn:M2}, we can write 
	\[
	(A^2)_{2,1}-(\theta_1+\theta_2)(A)_{2,1}=BA_1-\tilde{A_1}B+JB-(1+\theta_1+\theta_2)B = \alpha_1\alpha_2 J.
	\]

	Taking into account $\theta_1+\theta_2=-1$ and $JB=k_{12}J$, we have that 
	\begin{equation}\label{eq-ab-ba}
	BA_1-\tilde{A_1}B=(\alpha_1\alpha_2-k_{12})J.
	\end{equation}

	Multiply Eq. (\ref{eq-ab-ba}) by $\mathbf{j}$ to obtain the equality
	\begin{equation}
		\label{eqn:k12}
		\alpha_1\alpha_2-k_{12}=0,
	\end{equation}
	hence the required equality holds.
	
	Finally, we shall show Eq. (\ref{eq-like-starcomp}) (we prove only the first equation, the second 
	one follows similarly). 
%
	Start with the equation
	\[
	(\theta I-A_1)((\theta+1) I-(J-A_1)) = -A_1^2 - A_1 + (k_{11}-\theta)J + \theta(\theta+1)I.
	\]

	Since $\theta\in \{\theta_1,\theta_2\}$ and $\theta_1+\theta_2=-1$, 
	we see that $\theta(\theta+1)=-\theta_1\theta_2$. 
	Further, $k_{11}-\theta=\theta_0-k_{22}$, and by Eq. \eqref{eqn:k12}, we have
	$(\theta_0-k_{22})\alpha_2=k_{21}\alpha_1=\alpha_1^2\alpha_2$. 
	Thus
	\[
	(\theta I-A_1)((\theta+1) I-(J-A_1)) = -A_1^2 - A_1 + \alpha_1^2 J - \theta_1\theta_2 I = B^{\top}B,
	\]
	and the lemma follows.
\end{proof}

In Table~\ref{tab:fea} we observe that, apart from the Petersen cone, all feasible valency-arrays and spectra for biregular graphs in $\mathcal G(\theta_0,\theta_1,\theta_2)$ with $\theta_1+\theta_2 = -1$ have $n_1 = n_2$.
It is an interesting problem to decide whether this property follows from the spectrum in general except from the Petersen cone.


\section{Some nonexistence results} 
\label{sub:some_nonexistence_results}

In this final section we devote ourselves to showing the nonexistence of graphs corresponding to certain feasible valency-arrays and spectra.

Our first result here shows that there do not exist any graphs satisfying the conditions of Remark~\ref{rem:48}.

	\begin{theorem}\label{thm:48}
		There do not exist any graphs having
		valency-array $(24,24; 33,15)$ and spectrum $\{ [27]^1, [3]^{19}, [-3]^{28} \}$; or
		valency-array $(24,24; 32,14)$ and spectrum $\{ [26]^1, [2]^{27}, [-4]^{20} \}$.
	\end{theorem}
	\begin{proof}
		Suppose there exists a graph $\Gamma$ having the first valency-array and spectrum.
		The subgraph $\regSub$ induced on the vertices having valency $15$ is regular with valency $3$.
		Moreover, by interlacing, this subgraph has an eigenvalue of $-3$ with multiplicity at least $4$.
		Therefore $\regSub$ is four copies of $K_{3,3}$.
		Partition $V(\Gamma)$ into five parts consisting of the vertex sets of each $K_{3,3}$ and $V_1$.
		The quotient matrix
		\[
			Q = \begin{pmatrix}
				3 & 0& 0&0&12   \\ 
				0 & 3& 0&0&12   \\
				0 & 0 &3&0&12 \\
				0 & 0& 0&3&12	\\
				3 & 3&3&3&21
			\end{pmatrix}
		\]
		has eigenvalues $27$, $3$ (with multiplicity $3$), and $-3$.
		Since the interlacing of the eigenvalues of $Q$ with the eigenvalues of $\Gamma$ is tight, we know that the partition must be equitable.
		Therefore, each vertex of $V_1$ is adjacent to precisely three vertices of each $K_{3,3}$.
	    Let $X$ be a subgraph of $\Gamma$ induced on a set consisting of a copy of $K_{3,3}$ together with a vertex from $V_1$.
		Then $X$ can be one of two graphs, both of which have smallest eigenvalue less than $-3$.
		This gives a contradiction.
		
		Finally, suppose there exists a graph $\Gamma$ having the second valency-array and spectrum.
		Using Proposition~\ref{pro:bicomp}, one can see that the complement of $\Gamma$ has the first valency-array and spectrum.
		This completes the proof.
	\end{proof}

The following lemma is a simple application of the Cauchy-Schwarz inequality.

\begin{lemma}\label{lem:cs}
	Let $\Gamma$ be a connected $n$-vertex $k$-regular graph having a non-trivial eigenvalue $\theta$ with multiplicity $m$.
	Then $(k+m \theta)^2 \leqslant (n-1-m)(n k - k^2 - m \theta^2)$.
\end{lemma}
\begin{proof}
	Let $k$, $\theta$ (with multiplicity $m$), and $\eta_1, \dots, \eta_{n-1-m}$ be the eigenvalues of $\Gamma$.
	By Eqs.~\eqref{sums} we have $\sum_{i=1}^{n-1-m} \eta_i = -(k+m\theta)$ and $\sum_{i=1}^{n-1-m} \eta_i^2 = nk-k^2-m\theta^2$.
	Now apply the Cauchy-Schwartz inequality to obtain the desired result.
\end{proof}

\begin{theorem}\label{thm:100}
	There are no graphs having either 
	valency-array $(50,50; 69,33)$ and spectrum $\{ [57]^1, [7]^{24}, [-3]^{75} \}$; or
	valency-array $(50,50; 64,28)$ and spectrum $\{ [52]^1, [2]^{74}, [-8]^{25} \}$.
\end{theorem}
\begin{proof}
	The technique is the same for both sets of valency-arrays and spectra; we will deal only with first.
	Consider the $50$-vertex $9$-regular subgraph $\regSub$ induced by the subset of vertices $V_2$.
	By interlacing, $\regSub$ has eigenvalues $9$ with multiplicity $1$ and $-3$ with multiplicity $25$.
	Now apply Lemma~\ref{lem:cs} to deduce its nonexistence.
\end{proof}

Note that in our next result we show the nonexistence of a graph whose corresponding spectrum has the form considered in Section~\ref{sub:_theta_1_theta_2_1}.

\begin{theorem}\label{thm:44}
	There do not exist any graphs having valency-array
	$(22,22; 22, 7)$ and spectrum $\{ [19]^1, [2]^{22}, [-3]^{21} \}$.
\end{theorem}
\begin{proof}
	Let $\Gamma$ be a graph having the assumed valency-array and spectrum so that its valency partition has quotient matrix
	$\begin{pmatrix}
	18 & 4\\
	4 & 3
	\end{pmatrix}$.
	The subgraph $\regSub$ induced on the set $V_2$ is cubic.
	For vertices $x$ and $y$ in $V_2$ we have $\nu_{x,y} = 0$ if $x \sim y$ or $\nu_{x,y} = 1$ if $x \not \sim y$.
	Therefore, $\regSub$ cannot contain any triangles or any $4$-cycles.
	Moreover, $\regSub$ is a $22$ vertex cubic graph which, by interlacing, has second largest eigenvalue at most $2$.
	Using \texttt{Magma}~\cite{Magma} and \texttt{nauty}~\cite{nau}, we find that $\regSub$ must be one of four possible graphs.
	After adjoining five vertices in all possible ways to each of the four possible graphs we find that none of the resulting graphs has both smallest eigenvalue at least $-3$ and second largest eigenvalue at most $2$.
\end{proof}

In~\cite{Dam:2015qv} the authors examined a valency-array for a putative graph in $\mathcal G(30,3,-3)$ having four valencies.
We show that no such graph exists.

\begin{theorem}\label{thm:51}
	There are no graphs with
	valency-array $(1,30,5,15; 45,34,18,13)$ and spectrum $\{ [30]^1, [3]^{20}, [-3]^{30} \}$.
\end{theorem}

The technique of the proof of this result is similar to the techniques used in the result above. 
We therefore merely give a sketch of the proof.

\begin{proof}[Sketch of proof]
	Assume there exists a graph $\Gamma$ with the assumed valency-array and spectrum.
	Van Dam et al.~\cite{Dam:2015qv} determined much of the structure of $\Gamma$.
	In particular, it was shown that the valency partition of such a graph $\Gamma$ would be equitable with quotient matrix
	\[
	\begin{pmatrix}
	0 & 30 & 0 & 15\\
	1 & 25 & 3 & 5\\
	0 & 18 & 0 & 0\\
	1 & 10 & 0 & 2
	\end{pmatrix}.
	\]
	Starting with the valency-$2$ subgraph induced on $V_4$, one can apply similar techniques as given in the proofs above to determine the nonexistence of $\Gamma$.
\end{proof}


\section{Open Problems} 
\label{sec:open_problems}

Here we pose some open problems including ones that we have stated throughout the course of this article.

\begin{enumerate}[(I)]
	\item Does there exist a cone with valency-array $(1,234,54; 288,211,43)$ and spectrum $\{ [204]^1, [6]^{127}, [-6]^{161} \}$? (See Bridges and Mena \cite[Theorem 2.2]{brme}.)
	
	\item Classify the graphs in $\mathcal G(\theta_0,\theta_1,\theta_2)$ with $\dim(W(\Gamma)) = 9$.
	(See Remark~\ref{rem:WL}.)
	
	\item Is it true that, for any positive integer $C$, there exists a graph $\Gamma \in \mathcal G(\theta_0, \theta_1, \theta_2)$ with $\dim(W(\Gamma)) \geqslant C$? 
	
	\item Does there exist a pair of graphs having three eigenvalues such that both graphs have the same spectrum but different valency-arrays?
	(See Remark~\ref{rem:pairVA}.)
	
	\item Do there exist only finitely many connected biregular graphs with three distinct eigenvalues and bounded smallest eigenvalue? (See Theorem~\ref{thm:boundsec}.)
	
	\item Does the converse of Theorem~\ref{thm:2-des} hold?  
	
	\item Do there exist any graphs satisfying Theorem~\ref{thm:complement} or Proposition~\ref{pro:switch}?
	
	\item Do there exist anymore graphs with valency-array $(15,15; 14, 8)$ and spectrum $\{ [12]^1, [2]^{15}, [-3]^{14} \}$?
	(See Theorem~\ref{thm:30}.)
	
	\item Let $\Gamma$ be a biregular graph in $\mathcal G(\theta_0,\theta_1,\theta_2)$ with $\theta_1+\theta_2 = -1$.
	Suppose that $\Gamma$ is not the Petersen cone.
	Does it follow that $n_1 = n_2$?
	(See Lemma~\ref{lem:n1n2}.)

	\item De Caen~\cite[Problem 9]{Dam05} (also see \cite{Dam:2015qv}) asked whether graphs with three distinct eigenvalues have at most three distinct valencies.
	Is it possible to bound the number of distinct valencies of a graph with three distinct eigenvalues?
	
	\item Can one construct an infinite family of connected graphs having three distinct valencies and three distinct eigenvalues?
	(So far there are only finitely many known examples of such graphs.)
\end{enumerate}



\section{Acknowledgements } 
\label{sec:acknowledgements}
	The authors would like to thank the referees for their comments, which greatly helped improve the structure of the paper.

\appendix

\section{Feasible valency-arrays and spectra} 
\label{sec:feasible_parameters}

In this section we give a table of feasible valency-arrays and spectra for biregular graphs having precisely three eigenvalues and at most $100$ vertices.
Since we have a complete understanding in these cases, we omit valency-arrays and spectra that correspond to complete bipartite graphs (second largest eigenvalue equal to $0$) and valencies and spectra with second largest eigenvalue equal to $1$ that correspond to nonexistent graphs.
Finally we omit spectra that do not satisfy Theorem~\ref{thm:bell}.

In the comment column of Table~\ref{tab:fea}, we give some information about some of the valency-arrays and spectra.
If a graph corresponding to the valency-array and spectra exists, we try to give some indication of how it can be constructed.
If no such graph exists then we give a reference to a proof of its nonexistence.
Otherwise, if the existence of a graph is unknown, we may refer to a design related to the parameters such as a group divisible design (see Remark~\ref{rem:2-des}).

Table~\ref{tab:fea} was generated using \texttt{Magma}~\cite{Magma} and the associated code is available on request.

\setlength{\tabcolsep}{3pt}

\begin{center}
\begin{longtable}{lcllllllll}
	
\caption[Feasible valency-arrays and spectra]{Feasible valency-arrays and spectra for biregular graphs with three eigenvalues} \label{tab:fea} \\

\hline v & Spectrum & $n_1$ & $n_2$ & $k_1$ & $k_2$ & $k_{12}$ & Existence & Comment \\ \hline
\endfirsthead

\multicolumn{9}{c}%
{{\bfseries \tablename\ \thetable{} -- continued from previous page}} \\
\hline v & Spectrum & $n_1$ & $n_2$ & $k_1$ & $k_2$ & $k_{12}$ & Existence & Comment \\ \hline
\endhead

\hline \multicolumn{9}{r}{{Continued on next page}} \\ \hline
\endfoot

\hline \hline
\endlastfoot											

  11 & ${ [5]^1, [1]^{5}, [-2]^{5} }$     & $1$  & $10$ & $10$ & $4$  &  $10$ & 1 & Petersen cone \\
 14 & ${ [8]^1, [1]^{6}, [-2]^{7} }$      & $7$  & $7$  & $10$ & $4$  &  $4$  & 1 & Van Dam-Fano graph\\
 17 & ${ [8]^1, [2]^{6}, [-2]^{10} }$     & $1$  & $16$ & $16$ & $7$  &  $16$ & 2 & $\SRG(16,6,2,2)$ cone \\
 22 & ${ [14]^1, [2]^{7}, [-2]^{14} }$    & $14$ & $8$  & $16$ & $7$  &  $4$ & 1 & \cite{brme} \\
 29 & ${ [14]^1, [4]^{7}, [-2]^{21} }$    & $1$  & $28$ & $28$ & $13$ &  $28$ & 4 & $\SRG(28,12,6,4)$ cone\\
 30 & ${ [13]^1, [3]^{9}, [-2]^{20} }$    & $15$ & $15$ & $15$ & $7$  &  $3$ & 0 & \cite[Theorem 7]{Dam3ev} \\
 30 & ${ [18]^1, [3]^{8}, [-2]^{21} }$    & $15$ & $15$ & $22$ & $10$ &  $8$ & 0 & \cite[Theorem 7]{Dam3ev} \\
 30 & ${ [12]^1, [2]^{15}, [-3]^{14} }$   & $15$ & $15$ & $14$ & $8$  &  $4$ & $\geqslant 21$ & Theorem~\ref{thm:30} \\
 32 & ${ [14]^1, [4]^{8}, [-2]^{23} }$    & $16$ & $16$ & $16$ & $10$ &  $4$ & 0 & \cite[Theorem 7]{Dam3ev} \\
 36 & ${ [21]^1, [5]^{7}, [-2]^{28} }$    & $8$  & $28$ & $28$ & $18$ &  $21$ & 1 & \cite[Theorem 7]{Dam3ev} \\
 39 & ${ [14]^1, [2]^{23}, [-4]^{15} }$   & $12$ & $27$ & $17$ & $12$ &  $9$ & $\geqslant 120$ & Proposition~\ref{pro:biswitch}, \cite{Dam3ev} \\
 39 & ${ [20]^1, [2]^{22}, [-4]^{16} }$   & $12$ & $27$ & $26$ & $16$ &  $18$ & $\geqslant 120$ & Proposition~\ref{pro:biswitch}, \cite{Dam3ev} \\
 44 & ${ [22]^1, [6]^{8}, [-2]^{35} }$    & $28$ & $16$ & $24$ & $15$ &  $4$ & 0 & \cite[Theorem 7]{Dam3ev} \\
 44 & ${ [19]^1, [2]^{22}, [-3]^{21} }$   & $22$ & $22$ & $22$ & $7$  &  $4$ & 0 & Theorem~\ref{thm:44} \\
 45 & ${ [32]^1, [5]^{8}, [-2]^{36} }$    & $36$ & $9$  & $34$ & $16$ &  $4$ & 0 & \cite[Theorem 7]{Dam3ev} \\
 45 & ${ [20]^1, [2]^{26}, [-4]^{18} }$   & $9$  & $36$ & $32$ & $14$ &  $24$ & $\geqslant 9$ & Proposition~\ref{pro:switch}, \cite{Dam3ev} \\
 45 & ${ [14]^1, [2]^{27}, [-4]^{17} }$   & $9$  & $36$ & $20$ & $11$ &  $12$ & $\geqslant 9$ & Proposition~\ref{pro:switch}, \cite{Dam3ev}\\
 46 & ${ [15]^1, [3]^{20}, [-3]^{25} }$   & $1$  & $45$ & $45$ & $13$ &  $45$ & $78$ & $\SRG(45,12,3,3)$ cone \\
 48 & ${ [27]^1, [3]^{19}, [-3]^{28} }$   & $24$ & $24$ & $33$ & $15$ &  $12$ & 0 & Theorem~\ref{thm:48} \\
 48 & ${ [26]^1, [2]^{27}, [-4]^{20} }$   & $24$ & $24$ & $32$ & $14$ &  $12$ & 0 & Theorem~\ref{thm:48} \\
 50 & ${ [27]^1, [2]^{24}, [-3]^{25} }$   & $25$ & $25$ & $33$ & $9$  &  $9$ & 78 & Corollary~\ref{cor:2-des} \\
 54 & ${ [34]^1, [6]^{9}, [-2]^{44} }$    & $30$ & $24$ & $40$ & $19$ &  $12$ & 0 & \cite[Theorem 7]{Dam3ev} \\
 56 & ${ [21]^1, [3]^{24}, [-3]^{31} }$   & $20$ & $36$ & $27$ & $11$ &  $9$ & ? & $\GDD(20,36;9,5;0,2;10,2)$ \\
 57 & ${ [21]^1, [4]^{21}, [-3]^{35} }$   & $15$ & $42$ & $28$ & $16$ &  $14$ & ? & $(15,5,4)$-design\\
 57 & ${ [14]^1, [2]^{35}, [-4]^{21} }$   & $1$  & $56$ & $56$ & $11$ &  $56$ & 1 & $\SRG(56,10,0,2)$ cone\\
 66 & ${ [39]^1, [3]^{26}, [-3]^{39} }$   & $39$ & $27$ & $45$ & $13$ &  $9$ & 68 & Remark~\ref{rem:66} \\ 
 66 & ${ [33]^1, [6]^{18}, [-3]^{47} }$   & $54$ & $12$ & $34$ & $27$ &  $4$ & ? & \\
 66 & ${ [36]^1, [3]^{32}, [-4]^{33} }$   & $33$ & $33$ & $44$ & $20$ &  $16$ & ? & \\
 68 & ${ [34]^1, [2]^{46}, [-6]^{21} }$   & $48$ & $20$ & $37$ & $16$ &  $5$ & ? & \\
 69 & ${ [33]^1, [6]^{19}, [-3]^{49} }$   & $42$ & $27$ & $36$ & $26$ &  $9$ & ? & \\
 69 & ${ [24]^1, [2]^{48}, [-6]^{20} }$   & $21$ & $48$ & $32$ & $17$ &  $16$ & ? & \\
 70 & ${ [23]^1, [5]^{23}, [-3]^{46} }$   & $1$  & $69$ & $69$ & $21$ &  $69$ & ? & $\SRG(69,20,7,5)$ cone\\
 70 & ${ [25]^1, [3]^{40}, [-5]^{29} }$   & $10$ & $60$ & $33$ & $23$ &  $24$ & ? & $(10,4,8)$-design\\
 70 & ${ [30]^1, [2]^{48}, [-6]^{21} }$   & $14$ & $56$ & $48$ & $21$ &  $36$ & ? & $\GDD(14,56;36,9;24,22;7,2)$\\
 70 & ${ [22]^1, [2]^{49}, [-6]^{20} }$   & $14$ & $56$ & $32$ & $17$ &  $20$ & ? & \\
 74 & ${ [33]^1, [3]^{37}, [-4]^{36} }$   & $37$ & $37$ & $39$ & $15$ &  $9$ & ? & \\
 74 & ${ [38]^1, [2]^{50}, [-6]^{23} }$   & $34$ & $40$ & $48$ & $21$ &  $20$ & ? & \\
 78 & ${ [33]^1, [3]^{44}, [-5]^{33} }$   & $12$ & $66$ & $55$ & $25$ &  $44$ & ? & \\
 80 & ${ [47]^1, [7]^{19}, [-3]^{60} }$   & $40$ & $40$ & $53$ & $39$ &  $24$ & ? & \\
 80 & ${ [27]^1, [3]^{46}, [-5]^{33} }$   & $16$ & $64$ & $39$ & $21$ &  $24$ & $\geqslant 1$ & Proposition~\ref{pro:switch}, \cite{Dam3ev} \\
 80 & ${ [35]^1, [3]^{45}, [-5]^{34} }$   & $16$ & $64$ & $55$ & $25$ &  $40$ & $\geqslant 1$ & Proposition~\ref{pro:switch}, \cite{Dam3ev} \\
 80 & ${ [42]^1, [2]^{59}, [-8]^{20} }$   & $40$ & $40$ & $48$ & $34$ &  $24$ & ? & \\
 81 & ${ [33]^1, [6]^{23}, [-3]^{57} }$   & $27$ & $54$ & $42$ & $24$ &  $18$ & ? & \\
 81 & ${ [29]^1, [2]^{59}, [-7]^{21} }$   & $27$ & $54$ & $38$ & $20$ &  $18$ & ? & \\
 82 & ${ [27]^1, [6]^{24}, [-3]^{57} }$   & $1$  & $81$ & $81$ & $25$ &  $81$ & $\geqslant 1$ & $\SRG(81,24,9,6)$ cone\\
 84 & ${ [49]^1, [7]^{20}, [-3]^{63} }$   & $42$ & $42$ & $57$ & $37$ &  $24$ & ? & \\
 84 & ${ [44]^1, [4]^{36}, [-4]^{47} }$   & $36$ & $48$ & $56$ & $26$ &  $24$ & ? & \\
 84 & ${ [39]^1, [3]^{51}, [-6]^{32} }$   & $54$ & $30$ & $43$ & $27$ &  $10$ & ? & \\
 84 & ${ [44]^1, [2]^{62}, [-8]^{21} }$   & $42$ & $42$ & $52$ & $32$ &  $24$ & ? & \\
 85 & ${ [32]^1, [2]^{64}, [-8]^{20} }$   & $5$  & $80$ & $64$ & $28$ &  $64$ & 0 & $(5,4,48)$-design\\
 86 & ${ [51]^1, [3]^{51}, [-6]^{34} }$   & $68$ & $18$ & $54$ & $34$ &  $9$ & ? & $(18,9,16)$-design\\
 96 & ${ [44]^1, [4]^{42}, [-4]^{53} }$   & $48$ & $48$ & $52$ & $20$ &  $12$ & ? & \\
 96 & ${ [43]^1, [3]^{54}, [-5]^{41} }$   & $48$ & $48$ & $51$ & $19$ &  $12$ & ? & \\
 97 & ${ [24]^1, [4]^{45}, [-4]^{51} }$   & $1$  & $96$ & $96$ & $21$ &  $96$ & $\geqslant 1$ & $\SRG(96,20,4,4)$ cone \\
 98 & ${ [44]^1, [4]^{49}, [-5]^{48} }$   & $49$ & $49$ & $52$ & $28$ &  $16$ & ? &  \\
 99 & ${ [44]^1, [2]^{78}, [-10]^{20} }$  & $72$ & $27$ & $47$ & $32$ &  $9$ & ? & \\  
 100 & ${ [57]^1, [7]^{24}, [-3]^{75} }$  & $50$ & $50$ & $69$ & $33$ &  $24$ & 0 & Theorem~\ref{thm:100}\\
 100 & ${ [52]^1, [2]^{74}, [-8]^{25} }$  & $50$ & $50$ & $64$ & $28$ &  $24$ & 0 & Theorem~\ref{thm:100} \\				
                                                                                  
 \end{longtable}
 \end{center}

\bibliographystyle{myplain}
\bibliography{sbib}

\end{document}